\documentclass[11pt]{article}

\RequirePackage{amsthm,amsmath,amsfonts,amssymb}
\RequirePackage[numbers,sort&compress]{natbib}
\usepackage{hyperref}
\RequirePackage{graphicx}

\numberwithin{equation}{section}
\theoremstyle{plain}
\newtheorem{thm}{Theorem}[section]
\newtheorem{lem}[thm]{Lemma}
\newtheorem{cor}[thm]{Corollary}

\theoremstyle{definition}
\newtheorem{rem}[thm]{Remark}
\newtheorem{exam}[thm]{Example}

\renewcommand{\d}{{\rm d}}

\newcommand{\R}{\mathbb{R}}

\newcommand{\bE}{\mathbb{E}}

\begin{document}

\title{Quasi-ergodic theorems for Feynman--Kac semigroups and large deviation for additive functionals}

\author{Daehong Kim\footnote{The first named author is partially supported by a Grant-in-Aid for
		Scientific Research (C)  No.~23K03152 from Japan Society for the Promotion of Science.},~ Takara Tagawa and Aur\'elien Velleret}
	
	\date{\empty}

\maketitle
\begin{abstract}
We study the long-time behavior of an additive functional that takes into account the jumps of a symmetric Markov process. This process is assumed to be observed through a biased observation scheme that includes the survival to events of extinction and the Feynman--Kac weight by another similar additive functional.  
Under conditioning for the convergence to a quasi-stationary distribution and for two-sided estimates of the Feynmac--Kac semigroup to be obtained, we shall discuss general assumptions on the symmetric Markov process. For the law of additive functionals, we will prove a quasi-ergodic theorem, namely a conditional version of the ergodic theorem and a conditional functional weak law of large numbers. As an application, we also establish a large deviation principle for the mean ratio of additive functionals.

\begin{center}
	\textbf{Résumé}
\end{center}
Nous étudions le comportement à long terme d'une fonctionnelle additive qui tient compte des sauts d'un processus de Markov symétrique. Ce processus est supposé être observé au travers d'un schéma d'observation biaisé qui inclut la survie à des événements d'extinction et la pondération de Feynman--Kac via une autre fonctionnelle additive similaire.  
Sous des conditions permettant la convergence vers une distribution quasi-stationnaire et autorisant des estimations bilatérales du semigroupe de Feynmac--Kac, nous examinerons des hypothèses générales pour le processus de Markov symétrique. Pour la loi des fonctionnelles additives, nous prouverons un théorème quasi-ergodique, à savoir une version conditionnelle du théorème ergodique, et une loi faible des grands nombres, sous une version fonctionnelle et conditionnelle. Comme application, nous établissons également un principe de grande déviation pour le ratio moyen des fonctionnelles additives.
\end{abstract}



\section{Introduction}\label{sec:intro}  

Let $E$ be a locally compact separable metric space and $\mathfrak{m}$ be a positive Radon measure on $E$ with full topological support. 
Let $X=(\Omega, \mathcal M, \mathcal M_t, X_t, \mathbb{P}_x, \zeta)$ 
be a strong Markov process on $E$ with the augmented filtration $\mathcal M_t$ and the lifetime $\zeta:=\inf\{t>0 \mid X_t=\partial\}$, where $\partial$ is the cemetery point for $E$.
We use $\mathbb{E}_x$ to denote the expectation with respect to $\mathbb{P}_x$ for any $x \in E$.

Insightful properties of such a Markov process $X$ when conditioned to survive can already be derived by the study of the long-time behavior of the mean ratio of a continuous additive functional, as studied by several authors
(\cite{BR:1999, HYZ, Kaleta:2023, ZhangLiSong:2014}).
In this paper, we focus exclusively on processes $X$ that are almost surely killed.
Notably, Breyer and Roberts \cite{BR:1999} established when  
$\mathfrak{m}(E)<\infty$ and $X$ is positive $\lambda$-recurrent for some constant $\lambda \le 0$ that the following quasi-ergodic limit theorem holds: there exists a (probability) measure $\eta$ on $E$ such that for any $f \in L^1(E;\eta)$ and $x\in E$,
\begin{equation}\label{QEDintro}
	\lim_{t\to \infty}{\mathbb{E}}_x\left[\frac{1}{t}\int_0^tf(X_s){\rm d}s ~\Big|~ t<\zeta\right] =\int_E f(x)\eta({\rm d}x).
\end{equation}
The measure $\eta$ is often called a quasi-ergodic distribution of $X$. Zhang et al.~\cite{ZhangLiSong:2014} also established \eqref{QEDintro} under some conditions that are easier to check than positive $\lambda$-recurrence and $\mathfrak{m}$ a finite measure, and discussed the weak conditional law of large numbers for the continuous additive functional $\int_0^tf(X_s){\d}s$.

The quasi-ergodic distribution is closely linked to the long-time behavior of the transition semigroup $p_t$ of $X$ and to its principal eigenfunction (called the {\it ground state}). In this respect,
Takeda \cite{Takeda:2019-CPT} proved that if $X$ is $\mathfrak{m}$-symmetric 
and possesses the properties {\bf (I)}, {\bf (SF)} and {\bf (T)} (see Section 2 for their definitions), then the transition semigroup $p_t$ of $X$ is a compact operator on $L^2(E;\mathfrak{m})$ for any $t>0$, every eigenfunction has a bounded continuous version, and the ground state $\psi_0$ can be taken to be  strictly positive on $E$. Further, Takeda \cite{Takeda:2019-QSD} proved that if, in addition, $X$ is explosive, then 
the measure $\psi_0\mathfrak{m}/\int_E\psi_0 {\d}\mathfrak{m}$ is a unique quasi-stationary distribution of $X$ 
by showing that $\psi_0 \in L^1(E;\mathfrak{m})$.
Based on this setting, He et al. \cite{HYZ} established \eqref{QEDintro} even when $\mathfrak{m}(E)=\infty$ and also proved a weak conditional law of large numbers. 
\vskip 0.1cm
Throughout this paper, 
we assume that $X$ is $\mathfrak{m}$-symmetric in the sense of \eqref{msym}. 
Let $\mu=\mu^+-\mu^-$ be a signed smooth measure on $E$ 
whose associated continuous additive functional of $X$ is denoted by $A_t^\mu=A_t^{\mu^+}-A_t^{\mu^-}$. 
Let $F^+$ and $F^-$ be positive symmetric bounded measurable functions 
on $E\times E$ vanishing on the diagonal. Set $F:=F^+ -F^-$.
Then, it defines a purely discontinuous additive functional of $X$ by   
\begin{equation}
	A^{F}_t:=A_t^{F^+} 	-A_t^{F^-}:=\sum_{0<s\le t}F^+(X_{s-},X_s) - \sum_{0<s\le t}F^-(X_{s-},X_s)
\end{equation}
whenever it is summable. This additive functional often appears when considering the pure jump effects in Markov processes and, in a particular case, is thought to express the number of jumps in pure jump processes. 
It is natural to consider the following additive functional of the form $A_t^{\mu, F}:=A^\mu_t + A^F_t$. 
\vskip 0.1cm
In combination with the extinction, 
we are led to consider the following family of non-local Feynman--Kac transforms, indexed by a time $t$ which can be thought of as a final time for the biased action: for $x \in E$ and $\Lambda \in \mathcal M$,
\begin{equation}\label{genFKtansform}
	{\mathbb{P}}_{x:t}^{\mu, F}(\Lambda):={\mathbb{P}}_{x}\left(\Lambda \cdot W^{\mu, F}_t \right) \quad
	\text{where}~~ W^{\mu, F}_t := \exp \left(-A^{\mu, F}_t\right){\bf 1}_{\{t<\zeta\}}.
\end{equation}
It is useful to note that this bias extends the concept of extinction. With \eqref{genFKtansform}, we define a renormalized probability measure by
\begin{equation}\label{eq_def_PRxt}
	{\mathbb{P}}_{x|t}^{\mu, F}(\Lambda)
	=\frac{{\mathbb{P}}_{x:t}^{\mu, F}(\Lambda)}{{\mathbb{P}}_{x:t}^{\mu, F}(\Omega)}, \quad \text{for}~x \in E, ~\Lambda \in {\mathcal M}\,,
\end{equation} 
under the condition that the denominator is guaranteed to be finite. 
\vskip 0.1cm

Our main subject of this paper is to study a quasi-ergodic limit theorem for the additive functional \eqref{AF_AVG} below under the symmetric Markov process $X$ driven by the Feynman--Kac transform \eqref{genFKtansform}. More precisely, we shall prove under the Kato class condition related to the functional $A^{\mu, F}$
and the assumptions {\bf (A1)} $\sim$ {\bf (A3)} (see in Section \ref{sec:prelim} below) that there exist measures $\eta$ on $E$ and ${\mathcal J}$ on $E\times E$ such that 
the following limit holds
for any $V \in L^1(E;\eta)$, any measurable function $G$ on $E\times E$ vanishing on the diagonal and satisfying $\iint_{E\times E}|G(y,z)|{\mathcal J}({\d}y{\d}z) <\infty$, and any $x \in E$:
\begin{align}\label{muFQET}
\lim_{t\to \infty}{\mathbb{E}}_{x|t}^{\mu, F}\left[
\frac{1}{t}A_t^{V,G}\right] = \int_E V(x)\eta({\d}x) + \iint_{E\times E}G(x,y){\mathcal J}({\d}x{\d}y),
\end{align}  
where ${\mathbb{E}}_{x|t}^{\mu, F}$ denotes the expectation with respect to ${\mathbb{P}}_{x|t}^{\mu, F}$ and
\begin{align}\label{AF_AVG}
A_t^{V,G}:=\int_0^t V(X_s){\d}s +\sum_{0<s\le t}G(X_{s-},X_s).
\end{align}
The quasi-ergodic limit \eqref{muFQET}
gives us not only an extension for the Feynman--Kac scheme of  \eqref{QEDintro} but also the quasi-ergodicity caused by the pure jump effects for symmetric Markov processes that have not been previously discussed in other literature. 
Using \eqref{muFQET}, 
we further obtain the quasi-ergodic limit for second moment 
of $(1/t)\cdot\!A_t^{V,G}$ (see Lemmas~\ref{Moment2}, \ref{JumpWLLN})
and establish in Theorem~\ref{VGWLLN}
the weak conditional law of large numbers 
for the additive functional \eqref{AF_AVG} under ${\mathbb{P}}_{x|t}^{\mu, F}$: 
for any $\varepsilon >0$,
\begin{equation}
\lim_{t \to \infty}{\mathbb{P}}_{x|t}^{\mu, F}\left(\left|\frac{1}{t}A_t^{V,G} - \left(\int_E V(x)\eta(\d x)+\iint_{E\times E}G(y,z){\mathcal J}({\rm d}y{\rm d}z)\right)\right| \ge \varepsilon\right)=0.
\label{eq_VGWLLN}
\end{equation} 
It is known that any symmetric Markov process can be transformed into an ergodic process by some multiplicative functional (\cite[Chapter 6]{FOT}). To prove our results, we will apply this fact to transform a symmetric Markov process with Feynman--Kac weights into an ergodic process by a multiplicative functional called a ground state transform. Then, by using the Fukushima ergodic theorem for the semigroup of the transformed process, we give some ergodic limits for Feynman--Kac semigroups (Lemma \ref{QED1}), which play a crucial role in our proofs.
\vskip 0.1cm

The derived weak law of large number provides a new approach to derive a large deviation principle. Existing results on a large deviation principle for additive functionals are however numerous (see \cite{CT, DZ:1998, Remillard:2000, TT:2007,TT:2011}). In particular, Chen and Tsuchida \cite{CT} recently established a large deviation principle for pairs of continuous and purely discontinuous additive functionals under very general framework of symmetric Markov processes. Our Feynman--Kac scheme for the weak conditional law of large numbers  for $A_t^{V,F}$
helps to establish a large deviation principle for the mean ratio $(1/t)\cdot\!A_t^{V,G}$ 
by using a rate function with a more direct expression (see Theorem \ref{LDPforVF}). 
\vskip 0.1cm
The remainder of this paper is arranged as follows. In Section \ref{sec:prelim}, we give the setup with some assumptions. We also recall some definitions and known results that will be used in the rest of the paper. Section \ref{sec:QEforFK} is devoted to quasi-ergodic theorems (cf. Theorem~\ref{JumpErgodic})
and to the conditional functional weak law of large numbers (cf. Theorem~\ref{VGWLLN}). 
The large deviation principle for $(1/t)\cdot\!A_t^{V,G}$ is studied in Section \ref{sec:LDP}. In the last section, we give some examples satisfying our assumptions.

\section{Preliminaries}\label{sec:prelim}  

Let $E$ be a locally compact separable metric space with Borel $\sigma$-field $\mathcal B(E)$ and $\mathfrak{m}$ be a positive Radon measure on $E$ with full topological support. In the sequel, let $X=(\Omega, \mathcal M, \mathcal M_t, X_t, \mathbb{P}_x, \zeta)$ be a strong Markov process on $E$ with the lifetime $\zeta:=\inf\{t>0 \mid X_t=\partial\}$ of $X$. Denote by $\{p_t\}_{t\ge 0}$ and $\{R_\alpha\}_{\alpha \ge 0}$ the transition semigroup and resolvent of $X$, respectively: $p_tf(x)=\mathbb{E}_x[f(X_t)\, ;\, t<\zeta]$ and $R_\alpha f(x)=\int_0^\infty e^{-\alpha t}p_tf(x)\,{\d}t$ for $f \in \mathfrak{B}_b(E)$, where $\mathfrak{B}_b(E)$ denotes the space of all bounded Borel functions on $E$. We assume that $X$
is $\mathfrak{m}$-symmetric in the sense that:
\begin{align}\label{msym}
\int_E p_t f(x) \cdot g(x)\mathfrak{m}({\rm d}x)
=\int_E f(x) \cdot p_t g(x)\mathfrak{m}({\rm d}x), 
\quad \text{for any}~t>0~ \text{and}~ f, g \in L^2(E; \mathfrak m)\cap \mathfrak{B}_b(E).
\end{align} 

The transition kernel of $X$ is defined to be $p_t(x,{\d}y)={\mathbb{P}}_x(X_t \in {\d}y)$. 
It is known that $\{p_t\}_{t\ge 0}$ uniquely determines a strongly continuous Markovian semigroup $\{T_t\}_{t\ge 0}$ on $L^2(E;\mathfrak{m})$ (\cite[Lemma 1.4.3]{FOT}). Using this, we define the Dirichlet form $({\mathcal E},{\mathcal D}({\mathcal E}))$ on $L^2(E;\mathfrak{m})$ generated by $X$:
\begin{align*}
&{\mathcal D}({\mathcal E})=\left\{u \in L^2(E;\mathfrak{m})~\Big|~\lim_{t\to 0}\frac{1}{t}(u-T_tu, u)_{\mathfrak{m}} <\infty \right\} \\
&{\mathcal E}(u,v)=\lim_{t\to 0}\frac{1}{t}(u-T_tu, v)_{\mathfrak{m}}, \quad u,v \in {\mathcal D}({\mathcal E}).
\end{align*}

A Borel set $B \subset E$ is called $p_t$-invariant if $p_t(1_B\cdot f)=1_B\cdot p_tf$ 
$\mathfrak{m}$-a.e. for any $f \in L^2(E;\mathfrak{m})\cap  \mathfrak{B}_b(E)$ and $t>0$. 
We say that $X$ is irreducible ({\bf (I)} in short) if any $p_t$-invariant Borel set $B \subset E$ satisfies either $\mathfrak{m}(B)=0$ or $\mathfrak{m}(E\setminus B)=0$. The process $X$ (or $p_t$) is said to have strong Feller property ({\bf (SF)} in short) if $p_t(\mathfrak{B}_b(E)) \subset C_b(E)$ for any $t>0$, where $C_b(E)$ denotes the space of all bounded continuous functions on $E$. Further, the process $X$ is said to satisfy the absolute continuity condition ({\bf (AC)} in short) if the
transition kernel $p_t(x, {\rm d}y)$ of $X$ is absolutely continuous with respect to $\mathfrak{m}$, that is, there exists a transition density (heat kernel) $p_t(x,y)$ such that $p_t(x, {\rm d}y) =
p_t(x, y)\mathfrak{m}({\rm d}y)$ for any $x,y \in E$ and $t > 0$. We remark that {\bf (SF)} implies {\bf (AC)}. We say that $X$ possesses a tightness property ({\bf (T)} in short) if for any $\varepsilon >0$ there exists a compact set $K \subset E$ such that $
\sup_{x \in E}R_1{\bf 1}_{K^c}(x) \le \varepsilon$. There are many examples of symmetric Markov processes satisfying the properties {\bf (I)}, {\bf (SF)} and {\bf (T)} (see \cite{Takeda:2019-CPT,Takeda:2019-QSD}). By definition, $X$ is explosive if $\mathbb{P}_x(\zeta <\infty)>0$ holds for any $x\in E$. It is known that {\bf (T)} implies a strong recurrence if $X$ is conservative and a fast explosion if $X$ is explosive. In particular, {\bf (T)} implies that $X$ is almost surely killed, that is, $\mathbb{P}_x(\zeta <\infty)=1$ for any $x \in E$, if $X$ is explosive (cf. \cite[Section 4.1]{Takeda:2013-Tight}). 

A positive continuous additive functional (PCAF in short) $A_t$ of $X$ is called a PCAF in the strict sense if it admits a defining set $\Lambda$ with $\mathbb{P}_x(\Lambda)=1$ for all $x \in E$. In other words, the empty set $\emptyset$ can be taken as an exceptional set of $A_t$.
A Borel measure $\nu$ on $(E, \mathfrak{B}(E))$ is said to be smooth if $\nu$ charges no set of zero capacity and there exists a generalized nest $\{E_n\}$ of closed sets such that $\nu(E_n) < \infty$ for each $n \in \mathbb{N}$. 
We strengthen this condition under {\bf (AC)}. We may then extend $R_1$ to act on smooth Borel measure $\nu$ via $R_1\nu(x) = \int_E r_1(x, y)\, \nu(\d y)$, where $r_1$ is the unique jointly measurable function such that $R_1 f(x) = \int_E r_1(x, y)\, f(y) \mathfrak m(\d y)$ for any $f\in  \mathfrak{B}_b(E)$, (see \cite[Lemma 4.2.4 and Exercise 4.2.2]{FOT}).
Let $S_{00}$ be the collection of Borel measures $\nu$ on $E$ such that $\nu(E) <\infty$ and  $R_1\nu \in L^\infty(E;\mathfrak{m})$. We say that a Borel measure $\nu$ on $E$ is smooth in the strict sense if there exists a sequence $\{E_n\}$ of Borel sets increasing to $E$ such that ${\bf 1}_{E_n}\cdot \nu \in S_{00}$ for each $n\in \mathbb{N}$ and $\mathbb{P}_x (\lim_{n\to \infty}\tau_{E_n} \ge \zeta )=1$ for any $x \in E$, where $\tau_{E_n}:=\inf\{t>0 \mid X_t \notin E_n\}$. 

It is known that the family of equivalence classes of the set of PCAFs in the strict sense of $X$ and the family of positive smooth measures in the strict sense are in one-to-one correspondence under the Revuz correspondence (\cite[Theorem 5.1.7]{FOT}): for any $t>0$ and any Borel measurable function $f : E \to [0,\infty)$, there exists a positive smooth measure $\nu$ on $(E, \mathfrak{B}(E))$ such that
\begin{align}\label{Revuz}
	\int_{E}f(x)\nu({\d}x) = \lim_{t \to 0}\frac{1}{t}\int_E {\mathbb{E}}_x \left[\int_0^tf(X_s)\,{\d}A_s\right]\mathfrak{m}({\d}x). 
\end{align}
In this case, the measure $\nu$ is called the Revuz measure of $A_t$. We will write the PCAF $A_t$ of $X$ associated with $\nu$ as $A_t^\nu$ to emphasize the correspondence between $A_t$ and $\nu$.

\vskip 0.1cm
Let us denote by $\mathcal K(E\times E)$ the set of Borel measurable functions on $E\times E$ vanishing on the diagonal.
Let $(N(x, {\d}y), H_t)$ be a L\'evy system for $X$ (see \cite[Theorem A.3.21]{FOT}), that is, $N(x, {\d}y)$ is a kernel on $(E, {\mathfrak B}(E))$ and $H_t$ is a PCAF with bounded $1$-potential such that for any nonnegative function $K\in \mathcal K(E\times E)$
and for any $x\in E$, 
\begin{align*}
{\mathbb{E}}_x\left[\sum_{0<s\le t} K(X_{s-}, X_s),~ t<\zeta \right] 
={\mathbb{E}}_x\left[\int^t_0\!\int_{E} K(X_s, y) N(X_s,{\d}y) {\d}H_s,~ t<\zeta\right].
\end{align*}
$H$ being a PCAF, the jump times of $X$ do not contribute to the integral on the right-hand side of the above equation, so both $K(X_s, y)$ and $K(X_{s-}, y)$ are taken into account.
To simplify notation, we will write for such a kernel $N$:
\begin{equation}\label{eq_def_NK}
N[K f](x):= \int_{E} K(x, y)f(y) N(x, {\d}y),
\end{equation}
for any nonnegative Borel measurable function $f$ on $E$.
Let $\mu_H$ be the Revuz measure of the PCAF $H_t$. Then the jumping measure $J$ and the killing measure $\kappa$ of $X$ are given by $J({\d}x{\d}y)=\frac12  N(x, {\d}y)\mu_H ({\d}x)$ and $\kappa({\d}x)=N(x, \{\partial\})\mu_H({\d}x)$. 
These measures feature in the Beurling-Deny decomposition of $\mathcal E$: 
for $u, v\in {\mathcal D(\mathcal E)}$, 
$$
{\mathcal E}(u, v)={\mathcal E}^{c}(u, v)+\int_{E\times E}(\widetilde{u}(x)- \widetilde{u}(y))(\widetilde{v}(x)- \widetilde{v}(y))J({\d}x{\d}y)+\int_E \widetilde{u}(x)\widetilde{v}(x)\kappa({\d}x),
$$
where ${\mathcal E}^{c}$ is the strongly local part of ${\mathcal E}$ and $\widetilde{w}$ denotes the quasi continuous version of $w \in {\mathcal D}({\mathcal E})$ (\cite[Theorem 5.3.1]{FOT}). 
\vskip 0.1cm
A signed smooth measure $\nu:=\nu^+ -\nu^-$ on $E$ in the strict sense is said to be in the Kato class associated to $X$ ($\nu \in \mathcal S_K^1(X)$ in notation), if it satisfies that 
\begin{equation*}
\lim_{t\to 0}\sup_{x \in E}{\mathbb{E}}_x\left[\left|A_t^{\nu}\right|\right]=0,
\end{equation*}
where $|A_t^\nu|=A_t^{\nu^+} + A_t^{\nu^-}$. For various examples of measures belonging to the Kato class, including those induced by jumping functions, we refer to \cite{CKK, CS, KimKuwae:TAMS, KKT:2016}.
\vskip 0.2cm
Let $\mu$ be a signed smooth measure on $E$ such that $\mu \in \mathcal S_K^1(X)$. Let $F(x,y)$ be a symmetric (i.e., $F(x,y)=F(y,x)$) bounded
function in $\mathcal K(E\times E)$ satisfying
$N[|F|] \mu_H\in \mathcal S_K^1(X)$. Then, it defines 
an additive functional of $X$ by 
\begin{align*}
	A^{\mu, F}_t = A^{\mu}_t + A^{F}_t, \quad A^{F}_t = \sum_{0<s\le t}F(X_{s-},X_s)
	\end{align*}  
	for any $t\ge0$, in which the second component $A^{F}$ is purely discontinuous.
\vskip 0.2cm
Now we make the following assumptions: 
\begin{enumerate}
\item[{\bf (A1)}] $\mu_H = \mathfrak{m}$, so that $H_t = t$ for any $t<\zeta$.
\item[{\bf (A2)}] The process $X$ is explosive and possesses the properties {\bf (I)}, {\bf (SF)} and {\bf (T)}.
\end{enumerate}
\begin{rem}
{\rm 
\begin{enumerate}
	\item[(1)] By a suitable choice of a L\'evy system $(N(x,{\d}y), H_t)$, {\bf (A1)} captures the case of L\'evy processes (e.g. \cite[ Chapter VIII $\S$73]{Sha:Book}) and assuming $\mu_H({\d}x) = f(x) \mathfrak{m}({\d}x)$ for $f \in L^1(E; \mathfrak{m})$	would not extend the scope of our results.
	\item[(2)] It follows from \cite{Takeda:2019-CPT} that under {\bf (A2)}, the transition semigroup $p_t$ of $X$ is a compact operator on $L^2(E;\mathfrak{m})$ for any $t>0$. Thus, there is a unique minimizing function (called a ground state) $\psi_0 \in {\mathcal D}({\mathcal E})$ such that $\int_E \psi_0^2 {\d}\mathfrak{m}=1$ and 
		\begin{align}\label{GS0}
			{\mathcal E}(\psi_0, \psi_0)=\inf \left\{{\mathcal E}(u,u) : u \in {\mathcal D}({\mathcal E}), \int_E u^2{\d}\mathfrak{m}=1 \right\}.
		\end{align}
		In particular, $\psi_0$ can be taken to be strictly positive and has a bounded continuous version on $E$. Further, it was proved in \cite{Takeda:2019-QSD} that $\psi_0 \in L^1(E;\mathfrak{m})$ and the measure $\psi_0\mathfrak{m}/\int_E\psi_0 {\d}\mathfrak{m}$ is a unique quasi-stationary distribution of $X$.
\end{enumerate}
}
\end{rem}
\vskip 0.2cm
Let us give some examples of symmetric Markov processes satisfying the properties {\bf (I)}, {\bf (SF)} and {\bf (T)}:

\begin{exam}\label{ex2}
{\rm Let $X=(X_t, \mathbb{P}_x)$ be the symmetric $\alpha$-stable process in $\R^d$ with $0 <\alpha \le 2$ and $d\ge 1$, that is, $X$ is a conservative L\'evy process whose characteristic function is given by $\exp (-t|\xi|^{\alpha})~ (\xi \in \R^d)$. 
Let ${\mathcal D}$ be the family of open (not necessarily bounded and connected) sets in $\R^d$. Define
\begin{align}\label{Thin}
{\mathcal D}_0:=\left\{D \in {\mathcal D} ~\Big|~ \lim_{x \in D, |x|\to \infty}m(D\cap B(x,1))=0\right\},
\end{align}
where $m$ denotes the Lebesgue measure on $\R^d$ and $B(x,1)$ the open ball with center $x \in \R^d$
and radius $1$. A set $D \in {\mathcal D}_0$ is called thin at infinity in some literature. 
\begin{enumerate}
\item[(1)] For $D \in {\mathcal D}_0$, denote by $T_D:=\inf\{t>0 \mid X_t \notin D\}$, the first exit time of $X_t$ from $D$. Adjoin an extra point $\partial$ to $D$ and set $X_t^D= X_t$ for $T_D >t$
and $X_t^D= \partial$ for $T_D \le t$.
The process $X^D=(X_t^D, \mathbb{P}_x)$ is called the absorbing $\alpha$-stable process on $D$ and satisfies {\bf (I)}, {\bf (SF)} and {\bf (T)} (\cite[Lemma 2.1]{KT}).
\item[(2)] Let $V$ be a positive function on $\R^d$. Define the level set of $V$ by $D_M:=\{x \in \R^d \mid V(x) \le M\},$ for $M>0$.

Suppose that $V m \in \mathcal S_K^1(X)$ 
and $D_M \in {\mathcal D}_0$ for any $M>0$. Then, the subprocess $X^{V}$ of $X$ by $e^{-\int_0^tV(X_s){\d}s}$ satisfies {\bf (I)}, {\bf (SF)} and {\bf (T)} (cf. \cite[Theorem 5.2]{TTT:2017}).
\end{enumerate}
}
\end{exam}

Let us consider a non-local Feynman--Kac transform $\exp (A^{\mu, F}_t)$ by the additive functional $A_t^{\mu, F}$.
The Feynman--Kac weight acting on the trajectory up to time $t$ 
is then defined as follows:
\begin{equation*}
W^{\mu, F}_t:= {\exp} \left(-A^{\mu, F}_t\right)\mathbf{1}_{\{ t<\zeta \}}. 
\end{equation*}
With this specific form of $W^{\mu, F}_t$,
we we define the Feynman--Kac semigroup $\{p_t^{\mu, F}\}_{t\ge 0}$ by 
\begin{align}\label{FKuFmu}
p_t^{\mu, F} f(x):=
{\mathbb{E}}_x\left[W^{\mu, F}_t f(X_t)\right], \quad f \in \mathfrak{B}_b(E)\,.
\end{align} 
The facts that $\mu \in \mathcal S_K^1(X)$ and $N[|F|]\mathfrak m \in \mathcal S_K^1(X)$
entail that ${\mathbb{E}}_x[W^{\mu, F}_t]<\infty$ for any $x\in E$, $t>0$
(cf. \cite[Proposition 2.3]{CS03a}, as well as our later proof of Lemma~\ref{lem_AF_id}). 
The property is equivalent to ${\mathbb{P}}_{x:t}^{\mu, F}(\Omega)<\infty$ (recall \eqref{eq_def_PRxt}).
Thus $p_t^{\mu, F} f\in \mathfrak{B}_b(E)$ for any $f \in \mathfrak{B}_b(E)$ and $t>0$, and the conditional expectation $\mathbb E_{x|t}^{\mu, F}$ is well-defined. 
It follows from {\bf (A2)} that $p_t^{\mu, F}$ also satisfies {\bf (I)} and {\bf (SF)} (cf. \cite[Corollary 3.2]{KKT:2016} or \cite{KimKuwae:TAMS}). Thus, $p_t^{\mu,F}$ admits a symmetric integral kernel $p_t^{\mu,F}(x,y)$ such that $p_t^{\mu,F}(x,{\d}y)=p_t^{\mu,F}(x,y)\mathfrak{m}({\d}y)$ for any $x,y \in E$ and $t>0$.
\vskip 0.2cm

Let $\mathcal E^{\mu, F}$ be the symmetric bilinear form on $\mathcal D(\mathcal E)\times \mathcal D(\mathcal E)$ defined by 
\begin{align*}
\mathcal E^{\mu, F}(u,v)&:={\mathcal E}(u,v) + \int_Eu(x)v(x)\mu({\d}x) + \iint_{E\times E}u(x)v(y)\left(1-e^{-F(x,y)}\right)N(x,{\d}y)\mathfrak{m}({\d}x).
\end{align*}
In view of Stollmann--Voigt's inequality, $\mathcal E^{\mu, F}(u,v)$ is well-defined under $\mu, N[F]\mathfrak m\in\mathcal S_K^1(X)$
(cf. \cite[(4.5)]{CS03a}). Now, we define the bottom of the spectrum of $(\mathcal E^{\mu, F}, \mathcal D(\mathcal E))$ by
\begin{align}\label{specbot}
\lambda_0=\lambda_0(\mu,F):=\inf\left\{\mathcal E^{\mu, F}(u,u) : u \in \mathcal D(\mathcal E), \int_E u^2{\d}\mathfrak{m}=1 \right\}.
\end{align}
\vskip 0.2cm 
The following result on the existence of a ground state plays a crucial role in the present paper.

\begin{thm}{\rm (\cite[Corollary 1.1]{KKT:2016})}\label{GS}
There is a unique minimizing function (called a ground state) $\phi_0:=\phi_0^{\mu, F}$ in \eqref{specbot}, that is, there exists $\phi_0 \in \mathcal D(\mathcal E)$ such that $\int_E\phi_0^2(x)\mathfrak{m}({\d}x)=1$ and 
$\lambda_0=\mathcal E^{\mu, F}\left(\phi_0, \phi_0\right). $
\end{thm}
\noindent
Note that $\phi_0$ can be taken to be strictly positive due to {\bf (I)} of $\{p_t^{\mu, F}\}_{t\ge 0}$. However, 
unlike the case of $\psi_0$, we cannot guarantee the boundedness and integrability for $\phi_0$ without further assumption. 
\vskip 0.2cm
Now we further assume the following:
\begin{enumerate}
\item[{\bf (A3)}] The Feynman-Kac semigroup $\{p_t^{\mu, F}\}_{t\ge 0}$ is intrinsically ultracontractive ((IUC) in short), that is, there exists a constant $c_t >0$ such that
\begin{align*}
p_t^{\mu, F}(x,y) \le c_t \phi_0(x)\phi_0(y)
\end{align*}
for all $t>0$ and $x,y \in E$.
\end{enumerate}
\noindent
{\bf (A3)} ensures that 
$e^{\lambda_0 t} p_t^{\mu, F}(\phi_0) = \phi_0$ and that $\phi_0 \in L^1(E;\mathfrak{m}) \cap L^\infty (E;\mathfrak{m})$
(see the proof of \cite[Theorem 5.3]{TT:2013}). Further, by {\bf (SF)} of $p_t^{\mu, F}$, the ground state $\phi_0$ also has a continuous version.   
\vskip 0.2cm
The following is known as the Fukushima ergodic theorem (see \cite{Fuku:1982,Schmidt:2022,Takeda:2019-QSD}). 

\begin{thm}\label{FukuET}
Let $X$ be an $\mathfrak{m}$-symmetric irreducible conservative Markov process on $E$. Assume that $\mathfrak{m}(E)<\infty$. Then, for $f \in L^p(E;\mathfrak{m})$, $1<p<\infty$
\begin{align}\label{FETlim}
\lim_{t\to \infty}p_tf(x)=\frac{1}{\mathfrak{m}(E)}\int_E f\,{\rm d}\mathfrak{m} \quad \mathfrak{m}\text{-a.e. $\!x\in E$~and~in}~L^p(E;\mathfrak{m}).
\end{align}
\end{thm}
\vskip 0.2cm
\begin{rem}\label{rem to FET}
{\rm In addition to the assumptions of Theorem \ref{FukuET}, one can see that
\begin{enumerate}
	\item[(1)] If $\{p_t\}_{t\ge 0}$ is ultracontractive ((UC) in short), that is, $c_t:=\|p_t\|_{1,\infty} <\infty$, then \eqref{FETlim} holds for $f \in L^1(E;\mathfrak{m})$. Here $\|p_t\|_{1,\infty}$ denotes the operator norm of $p_t$ from $L^1(E;\mathfrak{m})$ to $L^\infty(E;\mathfrak{m})$. The conservativeness of $\{p_t\}_{t\ge 0}$ guarantees that $c_t$ is nonincreasing as a function of $t$.
	\item[(2)] If $X$ satisfies {\bf (AC)}, then \eqref{FETlim} holds for any $f \in L^\infty(E;\mathfrak{m})$ and the phrase $\mathfrak{m}$-a.e. $x \in E$ in \eqref{FETlim} can be strengthened to all $x \in E$.
\end{enumerate}
}
\end{rem}

\section{Quasi-ergodic theorems associated to Feynman--Kac transforms}\label{sec:QEforFK}  
Throughout this section, we continue to assume that $\mu \in \mathcal S_K^1(X)$ and  $N[F]\mathfrak{m}\in \mathcal S_K^1(X)$. We further suppose that the assumptions {\bf (A1)} $\sim$ {\bf (A3)} hold. For any measurable functions $V$ on $E$ and $G\in \mathcal K(E\times E)$, let
\begin{align*}
A_t^{V,G}:=A_t^V + A_t^G \quad \text{where}~~ A_t^V:=\int_0^t V(X_s) {\d}s, ~\, A_t^G:= \sum_{0<s\le t} G(X_{s-}, X_s).
\end{align*}
Our first aim is to have bounds on the fluctuations of the process $(1/t)\cdot A_t^{V,G}$ under the conditional law $\mathbb{P}^{\mu, F}_{x|t}$ when $t$ tends to infinity. 

\subsection{First computations for the expectation and variance}

For notational brevity, we will consider the compensated semigroup: 
\begin{equation*}
\begin{split}
\widehat p_t^{\mu, F} f(x) := e^{\lambda_0 t} p_t^{\mu, F} f(x)=e^{\lambda_0 t}\mathbb{E}_x\left[W_t^{\mu, F} f(X_t)\right], \quad f \in \mathfrak{B}_b(E).
\end{split}
\end{equation*}
The survival factor under the Feynman-Kac scheme will be denoted as
\begin{equation*}
\varPhi_t(x) 
:= \widehat p_t^{\mu, F} \mathbf{1}(x)
= e^{\lambda_0 t} \mathbb{E}_{x}\left[W^{\mu, F}_t\right].
\end{equation*}
The functions $\widehat p_t^{\mu, F} f$ and $\varPhi_t$ are to be asymptotically close to the ground state $\phi_0$ up to a constant factor  as we will see in the next subsection. 
\vskip 0.2cm
There is no difficulty in proving the following equalities by the Markov property for any nonnegative
measurable function $V$ on $E$:
\begin{align}
&{\mathbb{E}}_{x|t}^{\mu, F}\left[A_t^V\right]= \varPhi_t(x)^{-1}\int_0^t \widehat p^{\mu, F}_s(V\cdot \varPhi_{t-s})(x) {\d}s. \label{V_dist} \\	
&{\mathbb{E}}_{x|t}^{\mu, F}\left[\left(A_t^V\right)^2 \right]= 2 \varPhi_t(x)^{-1}
\int_0^t \int_0^{t-s} \widehat p^{\mu, F}_s\left(V\cdot \widehat p^{\mu, F}_r(V\cdot \varPhi_{t-s-r})\right)(x) {\d}\, r{\d}s. \label{V_dist1}
\end{align}
The analogous results for the observable of jumps are stated in the next lemma. 
Note, in general, that if $\{U_t\}_{t\ge 0}$ is a predictable nonnegative process on $\Omega$ and $K$ is a nonnegative function in $\mathcal K(E\times E)$, then it holds that
\begin{equation}\label{eq_prop_NK}
{\mathbb{E}}_x\left[\sum_{0<s\le t}U_s\cdot K(X_{s-},X_s)\right]	={\mathbb{E}}_x\left[\int_0^t U_s \cdot N[K](X_s){\d}s\right] 
\end{equation}
for any $t>0$ and $x\in E$ (cf. \cite[p.346]{Sha:Book}),
where we recall the definition of $N[K]$ in \eqref{eq_def_NK}.
\begin{lem}\label{LevysysFK}
Let $G$ be a  nonnegative function in $\mathcal K(E\times E)$. 
Then, for any $t>0$ and $x \in E$
\begin{align}
&{\mathbb{E}}_{x|t}^{\mu, F}\left[A_t^G\right]= \varPhi_t(x)^{-1}\int_0^t \widehat p^{\mu, F}_s \left(N\left[G e^{-F}\cdot \varPhi_{t-s}\right]\right)(x) {\d}s, \label{tg}\\
&{\mathbb{E}}_{x|t}^{\mu, F}\left[\left(A_t^G\right)^2\right]
= \varPhi_t(x)^{-1}\int_0^t \widehat p^{\mu, F}_s\big(N\left[G^2 e^{-F} \cdot \varPhi_{t-s}\right]\big)(x) {\d}s \label{tg2}\\
&\qquad \qquad \qquad \quad + 2\varPhi_t(x)^{-1}\int_0^t \int_0^{t-s} \widehat p^{\mu, F}_s \left(N\Big[G e^{-F}\cdot \widehat p^{\mu, F}_r\big(N\left[G e^{-F}\cdot \varPhi_{t-s-r}\right]\big)\Big]\right)(x) \,{\d}r {\d}s. \nonumber 
\end{align}
\end{lem}

\begin{proof}
First, we prove \eqref{tg}. By the Markov property 
\begin{align*}
{\mathbb{E}}_{x|t}^{\mu, F}\left[A_t^G\right] &
= \varPhi_t(x)^{-1}{\mathbb{E}}_{x}\left[\sum_{0<s\le t}e^{\lambda_0 s} W^{\mu, F}_{s}
\mathbf{1}_{\{ s\neq \zeta \}} G(X_{s-},X_s)e^{\lambda_0(t-s)}{\mathbb{E}}_{X_s}\left[W_{t-s}^{\mu, F}\right]\right] \\
&= \varPhi_t(x)^{-1}{\mathbb{E}}_{x}\left[\sum_{0<s\le t}e^{\lambda_0 s} W^{\mu, F}_{s-}\mathbf{1}_{\{ s\neq \zeta \}} G(X_{s-},X_s) e^{-F(X_{s-},X_s)}\cdot \varPhi_{t-s}(X_s)\right],
\end{align*}
which leads to define: 
$\widetilde G_{t-s}(x, y) 
:= G(x, y)e^{-F(x,y)}\cdot \varPhi_{t-s}(y).$
Due to \eqref{eq_prop_NK}, we have
\begin{equation*}
{\mathbb{E}}_{x|t}^{\mu, F}\left[A_t^G\right]= \varPhi_t(x)^{-1}{\mathbb{E}}_{x}\left[\int_0^t e^{\lambda_0 s}W^{\mu, F}_{s} \cdot N\left[\widetilde G_{t-s}\right](X_s) {\d}s\right].
\end{equation*}
Since the term under the integral is nonnegative, the equality \eqref{tg} is directly deduced using the Fubini--Tonelli theorem
and the identification of both $\widetilde G_{t-s}$ and $\widehat p^{\mu, F}_{s}$.  
\vskip 0.1cm	
For proving \eqref{tg2}, we distribute the expression in the left-hand side:
\begin{align*}
{\mathbb{E}}_{x|t}^{\mu, F}\left[\left(A_t^G\right)^2\right] ={\mathbb{E}}_{x|t}^{\mu, F}\left[\sum_{0<s\le t}G^2(X_{s-},X_s)\right]+2\,{\mathbb{E}}_{x|t}^{\mu, F}\left[\sum_{0<s<s'\le t}G(X_{s-},X_s)G(X_{s'-},X_{s'})\right].
\end{align*} 
From the result \eqref{tg}, the term of the sum with a squared contribution  at the same jump event gives rise to the first term on the right-hand side of \eqref{tg2}, namely,
\begin{equation}\label{sq_sum}
{\mathbb{E}}_{x|t}^{\mu, F}\left[\sum_{0<s\le t}G^2(X_{s-},X_s)\right]= \varPhi_t(x)^{-1}
\int_0^t \widehat p^{\mu, F}_s\left(N\left[G^2 e^{-F} \cdot \varPhi_{t-s}\right]\right)(x) {\d}s.
\end{equation}
What remains is the following sum over the pairs of ordered jumps, which we also handle thanks to the Markov property in combination with the first part of the proof for \eqref{tg}:
\begin{align*}
&2\,{\mathbb{E}}_{x|t}^{\mu, F}\left[\sum_{0<s<s'\le t}G(X_{s-},X_s)G(X_{s'-},X_{s'})\right]\\
&\quad = 2\,\varPhi_t(x)^{-1} {\mathbb{E}}_{x}\Bigg[\sum_{0<s\le t} e^{\lambda_0 s} W^{\mu, F}_{s-}\mathbf{1}_{\{ s\neq \zeta \}} G(X_{s-},X_s) e^{-F(X_{s-}, X_s)}\\
&\hspace{1.5cm} 
\cdot \mathbb{E}_x\Big[\sum_{s<s'\le t}e^{\lambda_0 (s'-s)} \frac{W^{\mu, F}_{s'-}}{W^{\mu, F}_s}
\mathbf{1}_{\{ s'\neq \zeta \}}
G(X_{s'-},X_{s'}) e^{-F(X_{s'-}, X_{s'})}\cdot \varPhi_{t-s'}(X_{s'})~\Big|~\mathcal M_{s}\Big]
\Bigg]\\
&\quad =2\,\varPhi_t(x)^{-1} {\mathbb{E}}_{x}\Bigg[\sum_{0<s\le t}
e^{\lambda_0 s} W^{\mu, F}_{s-}\mathbf{1}_{\{ s\neq \zeta \}} G(X_{s-},X_s) e^{-F(X_{s-}, X_s)}
\int_0^{t-s} \widehat p^{\mu, F}_r\left(N\left[G e^{-F}\cdot \varPhi_{t-s-r}\right]\right)(X_s) {\d}r
\Bigg].
\end{align*}
With a similar argument, we can deduce that
\begin{equation*}
\begin{split}
	&	2\,{\mathbb{E}}_{x|t}^{\mu, F}\left[
	\sum_{0<s<s'\le t}G(X_{s-},X_s)G(X_{s'-},X_{s'})\right] \\
	&\quad = 2\,\varPhi_t(x)^{-1}\int_0^t \widehat p^{\mu, F}_s\left(N\left[G e^{-F}\cdot 
	\int_0^{t-s} \widehat p^{\mu, F}_r\left(N\left[G e^{-F}\cdot \varPhi_{t-s-r}\right]\right){\d}r\right]\right)(x) {\d}s.
	\end{split}\end{equation*}
	This concludes the proof of the equality \eqref{tg2} by recalling \eqref{sq_sum}.  
\end{proof}

\subsection{Pointwise in time convergence results}
We are now ready to handle the expectation and variance of $(1/t)\cdot\!A_t^{V,G}$ in the limit where $t$ tends to infinity.
We shall exploit the so-called Doob transform that relies on the definition of a biased semigroup of the following form
\begin{equation*}
p_t^{\phi_0}f(x) 
:=\frac{1}{\phi_0(x)} \widehat{p}_t^{\mu, F}(\phi_0 f)(x), \quad f \in \mathfrak{B}_b(E).
\end{equation*}
In view of \cite[Lemma 6.3.2]{FOT}, we know that this expression defines the semigroup of a $\phi_0^2\mathfrak{m}$-symmetric irreducible and conservative Markov process on $E$, namely $X^{\phi_0}=(X_t, \mathbb{P}_x^{\phi_0})$ (whose extinction time is thus infinite, i.e. $\zeta^{\phi_0} \equiv \infty$).
The fact that $\phi_0$ is the ground state of $p_t^{\mu, F}$ implies that the following equality holds for any $t\ge 0$:
\begin{equation*}
{\mathbb{P}}_x^{\phi_0}(\Lambda_t)=\int_{\Lambda_t} L_{t}^{\phi_0}(\omega){\mathbb{P}}_x({\d}\omega), \quad \Lambda_t \in \mathcal M_t,
\end{equation*}
where $L^{\phi_0}_t$ is a multiplicative functional (a so-called ground state transform) defined by 
\begin{equation*}
L_t^{\phi_0} = e^{\lambda_0 t}\frac{\phi_0(X_t)}{\phi_0(x)} W^{\mu, F}_t.
\end{equation*}
\smallskip

We note that, under {\bf (A2)} and {\bf (A3)},  $X^{\phi_0}$ satisfies {\bf (AC)} and that $\{p_t^{\phi_0}\}_{t\ge 0}$ is ultracontractive. So, by applying Theorem \ref{FukuET} (with Remark \ref{rem to FET}) to $X^{\phi_0}$, we see that for any $f \in L^1(E;\phi_0\mathfrak{m})$,
\begin{equation}\label{FKFET}
\lim_{t\to \infty}\widehat{p}_t^{\mu, F} f(x)=\lim_{t\to \infty}\phi_0(x)p_t^{\phi_0}\left(\frac{f}{\phi_0}\right)(x) 
=\phi_0(x)\int_E f(y)\phi_0(y)\mathfrak{m}({\d}y) \quad \text{for any}~x\in E.
\end{equation}
In particular, 
we obtain the following convergence:
\begin{equation}\label{vFt_cvg}
\lim_{t\rightarrow \infty} \varPhi_t(x) 
= \lim_{t\rightarrow \infty}\widehat p_t^{\mu, F} \mathbf{1}(x)
= \phi_0(x) \int_E \phi_0(y)\mathfrak{m}({\d}y)\quad \text{for any}~x\in E.
\end{equation}
Moreover, the following upper-bound holds
with $c_t := \|p^{\phi_0}_t\|_{1,\infty}<\infty$ for any $t>0$:
\begin{equation}\label{eq_phi_ratio}
\left\|\phi_0^{-1}\cdot \varPhi_t\right\|_{L^\infty(E;\phi_0^2 \mathfrak m)}
\le c_t\cdot \int_E \phi_0 {\rm d}\mathfrak m< \infty,
\end{equation}
where $\|p_t^{\phi_0}\|_{1,\infty}$ denotes the operator norm of $p_t^{\phi_0}$ from $L^1(E;\phi_0^2\mathfrak{m})$ to $L^\infty (E;\phi_0^2\mathfrak{m})$. In fact, since $\phi_0$ is itself bounded, $\sup_{t\ge 1} \|\varPhi_t\|_{L^\infty(E ; \phi_0^2 \mathfrak m)} < \infty$.
\vskip 0.1cm
Let $(N^{\phi_0}(x,{\d}y), H_t)$ be a L\'evy system for $X^{\phi_0}$. Then, it holds by {\bf (A1)} and \cite[Lemma 2.9]{CFTYZ} that
\begin{equation*}
N^{\phi_0}(x,{\d}y):=\frac{\phi_0(y)}{\phi_0(x)}e^{-F(x,y)}N(x,{\d}y), \quad H_t=t.
\label{def_Nphi}
\end{equation*}
Define the measures $\nu_{\phi_0}$ on $E$ and ${\mathcal J}_{\phi_0}$ on $E\times E$, respectively, by $\nu_{\phi_0} := \phi_0 \mathfrak{m} /\int_E \phi_0 {\d}\mathfrak{m}$ and
\begin{equation*}
{\mathcal J}_{\phi_0}({\rm d}x{\rm d}y):=\phi_0(x)\phi_0(y)e^{-F(x,y)}N(x,{\rm d}y)\mathfrak{m}({\rm d}x) = N^{\phi_0}(x, {\d}y) \phi_0^2(x)\mathfrak{m}({\rm d}x).
\end{equation*}
The symmetry of the jumping measure ${\mathcal J}_{\phi_0}$ translates as follows in our notations, 
for any nonnegative Borel measurable functions $f, g$ on $E$
and any nonnegative function $K\in \mathcal K(E\times E)$:
\begin{equation}\label{eq_sym_Jphi}
\int_E f(x)\, N^{\phi_0}[K\cdot g](x)\; \phi_0^2(x)\, \mathfrak{m}({\rm d}x)
= 	\int_E g(y)\, N^{\phi_0}[K^T\cdot f](y)\; \phi_0^2(y)\, \mathfrak{m}({\rm d}y)\,,
\end{equation}
with the transpose kernel denoted $K^T$, such that $K^T(x, y) = K(y, x)$ for any $x, y\in E$.
In particular, 
the $L^1(E; \phi_0^2\mathfrak{m})$-norm of $N^{\phi_0}[K]$ coincides with the one of $N^{\phi_0}[K^T]$. 
\vskip 0.1cm
First, we give the pointwise in time convergence results for the Feynman--Kac semigroup $\widehat p_t^{\mu, F}$. The assertions $(1)$ and $(2)$ of the following lemma are originally due to \cite[Theorem 3.2]{ZhangLiSong:2014} 
for $f, g \in L^1(E;\mathfrak{m})$ when $\mu=F=0$ and $\mathfrak{m}$ is a finite measure on $E$ 
(also, \cite[Theorem 3.1]{HYZ} for $f, g \in \mathfrak{B}_b(E)$ when $\mu=F=0$).    
Nonetheless, we need to modify these results 
for $f, g$  in $L^1(E;\phi_0^2\mathfrak{m})$ 
or $ L^1(E;\phi_0\mathfrak{m})$ under the Feynman--Kac scheme, 
because it is essential to derive our quasi-ergodic theorem
for the discontinuous additive functional induced by a jumping function $F$. 
We give the proofs for reader's convenience
and consider the new assertions $(3)$ and $(4)$ involving the L\'evy system.

\begin{lem}\label{QED1}
\begin{enumerate}
\item[{\rm (1)}]
For any $f \in L^1(E;\phi_0^2\mathfrak{m})$, $g \in L^1(E;\phi_0\mathfrak{m})$, $x\in E$ and $0 < p <1$, we have
\begin{equation*}
	\lim_{t\to \infty}\varPhi_t(x)^{-1}\,\widehat p^{\mu, F}_{pt}\left(f\cdot \widehat p^{\mu, F}_{(1-p)t}g\right)(x)
	=\int_E f\phi_0^2\d\mathfrak{m} \cdot \int_E g\d\nu_{\phi_0}.
\end{equation*}
\item[{\rm (2)}] For any $f, g \in L^1(E;\phi_0^2\mathfrak{m})$, $x \in E$ and $0 < p<q <1$, we have
\begin{equation*}
	\lim_{t\to \infty}\varPhi_t(x)^{-1}\, \widehat p^{\mu, F}_{p t}\left(f\cdot \widehat p^{\mu, F}_{(q-p)t}\left(g\cdot \varPhi_{(1-q) t}\right)\right)(x)=\int_E f\phi_0^2\d\mathfrak{m}\cdot \int_E g\phi_0^2\d\mathfrak{m}.
\end{equation*}

\item[{\rm (3)}]
For any $G\in \mathcal K(E\times E)$
such that $N^{\phi_0}[|G|]\in L^1(E;\phi_0^2\mathfrak{m})$, $g \in L^1(E;\phi_0\mathfrak{m})$, $x \in E$ and $0 < p<1$, we have
\begin{equation*}
	\lim_{t\to \infty}
	\varPhi_t(x)^{-1}\,\widehat p^{\mu, F}_{pt}\left(N\left[G e^{-F} \cdot\widehat p^{\mu, F}_{(1-p)t}g\right]\right)(x)  
	= \iint_{E\times E}G(y,z){\mathcal J}_{\phi_0}({\d}y{\d}z)\cdot \int_E g \d\nu_{\phi_0}.
\end{equation*}

\item[{\rm (4)}]
For any $G, K\in \mathcal K(E\times E)$
such that $N^{\phi_0}[|G|], N^{\phi_0}[|K|] \in L^1(E;\phi_0^2\mathfrak{m})$, $g \in L^1(E; \phi_0 \mathfrak{m})$, $x \in E$ and $0 < p<q<1$, we have,
\begin{align*}
	&\lim_{t\to \infty}\varPhi_t(x)^{-1}\,\widehat p^{\mu, F}_{pt}\Big(N\Big[G e^{-F} \cdot \widehat p^{\mu, F}_{(q-p)t}\left(N\left[Ke^{-F} \cdot \widehat p^{\mu, F}_{(1-q)t}g\right]\right)\Big]\Big)(x) \\
	&\qquad \quad = \iint_{E\times E}G(y,z){\mathcal J}_{\phi_0}({\d}y{\d}z) \cdot \iint_{E\times E} K(y,z){\mathcal J}_{\phi_0}({\d}y{\d}z)\cdot \int_E g {\rm d}\nu_{\phi_0}.
\end{align*}
\end{enumerate}
\end{lem}

\begin{rem}\label{remQSDexpress}{\rm 
Note the following equality for $0<p<q<1$ to better interpret the terms on the left-hand side of (1) and (2) of Lemma \ref{QED1} for which convergence is shown:
\begin{equation*}
	\begin{split}
		&\varPhi_t(x)^{-1}\,
		\widehat p^{\mu, F}_{pt}\left(f\cdot \widehat p^{\mu, F}_{(1-p)t}g\right)(x)
		= {\mathbb{E}}_{x|t}^{\mu, F}\left[f(X_{pt})g(X_t)\right],
		\\	
		&\varPhi_t(x)^{-1}\,
		\widehat p^{\mu, F}_{p t}\left(f\cdot 
		\widehat p^{\mu, F}_{(q-p)t}\left(g\cdot \varPhi_{(1-q) t}\right)\right)(x)
		= {\mathbb{E}}_{x|t}^{\mu, F}\left[f(X_{pt})g(X_{qt})\right].
	\end{split}
\end{equation*}
The left-hand terms of (3) and (4) of Lemma \ref{QED1} are not so directly interpretable, but they concern the average contribution of jumps occurring at times $pt$ and $qt$, respectively. The freedom in the function $g$ is exploited in the proof.}
\end{rem}

\begin{rem}\label{remQSDPenal}{\rm 
\begin{enumerate}
	\item[{\rm (1)}] By applying $f\equiv 1$ in Lemma \ref{QED1}(1), we see that for any $g \in L^1(E;\phi_0\mathfrak{m})$ and $x\in E$
	\begin{equation}\label{ExpQE}
		\lim_{t\to \infty}{\mathbb{E}}_{x|t}^{\mu, F}\left[g(X_t)\right] =\int_E g\, \d\nu_{\phi_0}. 
	\end{equation}
	This means that $\nu_{\phi_0}$ is a Yaglom limit, in particular, the unique quasi-stationary distribution of $X$ under ${\mathbb{P}}_{x|t}^{\mu, F}$ (cf. \cite{MV:2012}). Takeda and Tawara \cite{TT:2013} also obtained a result similar to the above by showing that the Feynman--Kac semigroup $\{p_t^{\mu, F}\}_{t\ge 0}$ is quasi-ergodic. It can be shown in a similar way as in \cite[Corollary 2]{KnoP}  that the limit \eqref{ExpQE} indeed converges exponentially: there exist $t_0>0$ and the constants $C, \gamma >0$ such that 
	\begin{equation*}
		\left|{\mathbb{E}}_{x|t}^{\mu, F}[g(X_t)]-\int_E g\, \rm d\nu_{\phi_0}\right| \le Ce^{-\gamma t} 
	\end{equation*} 
	holds for all $t\ge t_0$ and $x \in E$. Kaleta and Schilling \cite{Kaleta:2023} recently obtained
	the exponential quasi-ergodicity of the strong Feller semigroups with the finiteness of their heat content.
	
	\item[{\rm (2)}] For any $f \in L^1(E;\phi_0^2\mathfrak{m})$ and $x\in E$, the following limiting behavior can be shown with a similar approach as in Lemma \ref{QED1}(1):
	\begin{align*}
		\lim_{s\to \infty}\lim_{t\to \infty}{\mathbb{E}}_{x|t}^{\mu, F}[f(X_s)]
		&=\lim_{s\to \infty}\lim_{t\to\infty}\varPhi_t(x)^{-1}\,				\widehat p^{\mu, F}_{s}\left(f\cdot \widehat p_{t-s}^{\mu, F} \mathbf{1}\right)(x) \\
		&= \lim_{s\to \infty}\phi_0(x)^{-1} \widehat p^{\mu, F}_{s}(f \phi_0)(x)=  \lim_{s\to \infty}p^{\phi_0}_{s}f(x)=\int_E f\phi_0^2\d\mathfrak{m}.
	\end{align*}
	From this, we see that a Yaglom limit is different from the limit of the Feynman--Kac penalization.
\end{enumerate}
}
\end{rem}

\begin{proof}[Proof of Lemma \ref{QED1}] 
(1):~First, we assume that $f$ and $g$ are nonnegative. For each fixed $r_0 >0$, let 
\begin{equation}\label{MIN_PHI}
h_{r_0}^{(g)}(x):=\frac{1}{\phi_0(x)}\inf_{r\ge r_0}\widehat p_r^{\mu, F} g(x)
=\inf_{r\ge r_0}p_r^{\phi_0}\left(\frac{g}{\phi_0}\right)(x).
\end{equation}
We then exploit that $\widehat p_{(1-p)t}^{\mu, F} g(x) \ge h_{r_0}^{(g)}(x)\phi_0(x)$ for each $r_0>0$ such that $(1-p)t \ge r_0$.
Recalling that  $\{p_t^{\phi_0}\}_{t\ge 0}$ is ultracontractive, $\|p_t^{\phi_0}f\|_{L^\infty(E;\phi_0^2\mathfrak{m})} <\infty$ for any $f \in L^1(E;\phi_0^2\mathfrak{m})$ and for all $t \ge 0$. Then,
\begin{equation}
\label{eq_L1_control}
\begin{split}
	\int_E fh_{r_0}^{(g)} \phi_0^2\d\mathfrak{m} 
	&\le \int_E f \cdot p_{r_0}^{\phi_0}\left(\frac{g}{\phi_0}\right) 
	\cdot\phi_0^2\d\mathfrak{m} \\
	&=\int_E p_{r_0}^{\phi_0}f \cdot g \phi_0\d\mathfrak{m} \le \left\|p_{r_0}^{\phi_0}f\right\|_{L^\infty(E;\phi_0^2\mathfrak{m})}\int_E g\phi_0\d\mathfrak{m} <\infty,
	\end{split}\end{equation}
	that is, $f\,h_{r_0}^{(g)} \in L^1(E;\phi_0^2\mathfrak{m})$. 
	Thanks to \eqref{FKFET} and \eqref{vFt_cvg},
	this fact implies the following:
	\begin{equation}\label{cvinf_rpfg}
\liminf_{t\to \infty}\varPhi_t(x)^{-1}\,
\widehat p^{\mu, F}_{pt}\left(f\cdot \widehat p_{(1-p)t}^{\mu, F} g\right)(x)
\ge \liminf_{t\rightarrow \infty}\varPhi_t(x)^{-1}\,\widehat p^{\mu, F}_{pt}\left(f h_{r_0}^{(g)} \phi_0\right)(x)
=\frac{\int_E fh_{r_0}^{(g)}\phi_0^2\d\mathfrak{m}}{\int_E \phi_0\d\mathfrak{m}}.
\end{equation}
On the other hand, since $g/\phi_0 \in L^1(E; \phi_0^2 \mathfrak{m})$,
$h_{r_0}^{(g)}(y) \stackrel{r_0 \to \infty}{\longrightarrow} 
\int_E g\phi_0\d\mathfrak{m}$ for any $y\in E$. 
Hence, letting $r_0 \to \infty$, by the monotone convergence theorem,
\begin{equation}\label{Infest}
\begin{split}
	\liminf_{t\to \infty}\varPhi_t(x)^{-1}\,
	\widehat p^{\mu, F}_{pt}\left(f\cdot \widehat p_{(1-p)t}^{\mu, F} g\right)(x)
	&\ge\frac{\int_E f\phi_0^2\d\mathfrak{m} 
		\cdot\int_E g\phi_0\d\mathfrak{m}}{\int_E \phi_0 \d\mathfrak{m}} =\int_E f\phi_0^2\d\mathfrak{m} \cdot \int_E g\d\nu_{\phi_0}. 
\end{split}
\end{equation}

We note that the limit is already known thanks to \eqref{FKFET} provided $g = \phi_0$, because $f\cdot \widehat p^{\mu, F}_{(1-p)t}\phi_0	= f \phi_0 \in L^1(E; \phi_0 \mathfrak{m})$ does not depend on $t$. From this observation, the converse inequality in the limit is deduced as follows.
For $n \ge 1$, let $g_n = \min \{g, n\phi_0\}$. Applying $n\phi_0-g_n$ instead of $g$ in the left-hand side of \eqref{Infest}, we obtain:
\begin{align*}
\liminf_{t\to \infty}\varPhi_t(x)^{-1}\,\widehat p^{\mu, F}_{pt}\left(f\cdot \widehat p_{(1-p)t}^{\mu, F} (n\phi_0-g_n)\right)(x) &=\liminf_{t\to \infty}\varPhi_t(x)^{-1}\left\{n\widehat p^{\mu, F}_{pt}(f\phi_0)-\widehat p^{\mu, F}_{pt}\left(f\cdot \widehat p_{(1-p)t}^{\mu, F} g_n\right)(x)\right\} \\
&=\frac{n\int_E f\phi_0^2\d \mathfrak{m}}{\int_E \phi_0\d \mathfrak{m}}-\limsup_{t\to \infty}\varPhi_t(x)^{-1}\widehat p^{\mu, F}_{pt}\left(f\cdot \widehat p_{(1-p)t}^{\mu, F} g_n\right)(x),
\end{align*}
while the right-hand side of \eqref{Infest} becomes
\begin{equation}\label{rhsc} 
\int_E f\phi_0^2\d\mathfrak{m} \cdot \int_E (n \phi_0 - g_n)
\d\nu_{\phi_0}= \dfrac{n\int_E f\phi_0^2\d\mathfrak{m}}{\int_E \phi_0 \mathfrak{m}} 
- \int_E f\phi_0^2\d\mathfrak{m} \cdot\int_E g_n\d\nu_{\phi_0}.
\end{equation}
Thus, \eqref{Infest} implies the following:
\begin{equation}\label{SupEst}
\limsup_{t\to \infty}\varPhi_t(x)^{-1}\widehat p^{\mu, F}_{pt}\left(f\cdot \widehat p_{(1-p)t}^{\mu, F} g_n\right)(x) 
\le \int_E f\phi_0^2\d\mathfrak{m} \cdot \int_E g_n\d\nu_{\phi_0}.
\end{equation}
By the monotone convergence theorem, 
the $L^1(E;\phi_0\mathfrak{m})$-norm of $g-g_n$ goes to 0 as $n\to \infty$.
By adapting the argument in \eqref{eq_L1_control}
to $g-g_n$ instead of $g$, we deduce for any $p\in (0,1)$
that the $L^1(E;\phi_0^2 \mathfrak m)$-norm of 
$\xi^p_{n, t}
=\phi_0^{-1}\, f\cdot \widehat p_{(1-p)t}^{\mu, F} (g - g_n)$
converges to zero as $n\to \infty$
uniformly  in $t\ge r_0/(1-p)$:
\begin{equation*}
	\int_E \xi^p_{n, t} \phi_0^2\d\mathfrak{m} =\int_E  \phi_0^{-1}\widehat p_{(1-p)t}^{\mu, F} \left(g - g_n\right)\cdot f\, \phi_0^2\d\mathfrak{m} 
	\le c_{r_0} \left\|f\right\|_{L^1(E;\phi_0^2\,\mathfrak{m})}
	\int_E \left(g-g_n\right)\phi_0\d\mathfrak{m} ~\longrightarrow~ 0 \quad \text{as}~~n\to \infty.
\end{equation*}
Hence,
one has by (UC) of $\{p_t^{\phi_0}\}_{t\ge 0}$,
	for any $t\ge r_0/(1-p)$,
\begin{equation}
	\label{eq_L1_approx}
	\begin{split}
		\widehat p^{\mu, F}_{pt}\left(f\cdot \widehat p_{(1-p)t}^{\mu, F} \left(g -g_n\right)\right)(x) 	
		&= \phi_0(x) \cdot p^{\phi_0}_{pt}\left(\xi^p_{n, t} \right)(x) \le c_{r_0}^2\,\phi_0(x) 
			\left\|f\right\|_{L^1(E;\phi_0^2\,\mathfrak{m})}
			\int_E \left(g-g_n\right)\phi_0\d\mathfrak{m},
	\end{split}
\end{equation}
in which the right-hand side converges to 0 as $n\to \infty$ independently of $t$.
By linearity of the left-hand side regarding $g$, we deduce that \eqref{SupEst} extends for $g$ as $n\to \infty$. 
Now the assertion is obtained for nonnegative pairs $f$ and $g$. It directly extends to signed functions by linearity. 
\vskip 0.1cm
(2): The proof is similar to that of (1). First, assume that $f$ and $g$ are nonnegative. 
For each $r_0 >0$ such that both $(1-q)t \ge r_0$ and $(q-p)t \ge r_0$, 
the inequality below is deduced by virtue of \eqref{MIN_PHI}, i.e. the definition of $h_{r_0}^{(\rho)}$:
\begin{equation}\label{FKETLpq}
\varPhi_t(x)^{-1}\,
\widehat p^{\mu, F}_{pt}\left(f\cdot \widehat p_{(q-p)t}^{\mu, F} \left(g\cdot \varPhi_{(1-q)t}\right)\right)(x)
\ge \varPhi_t(x)^{-1}\,
\widehat p^{\mu, F}_{pt}\left(f \phi_0 h_{r_0}^{(\rho)}\right)(x),
\end{equation}
where $\rho(x) = g(x)\phi_0(x) h_{r_0}^{(\mathbf{1})}(x)$. By the definition of $h_{r_0}^{(\rho)}$ and the fact that $\{p_t^{\phi_0}\}_{t\ge 0}$ satisfies (UC), 
one can check that 
$f h_{r_0}^{(\rho)}\in L^1(E;\phi_0^2\mathfrak{m})$ with a norm upper-bounded by 
$c_{r_0}^2 \left\|f\right\|_{L^1(E;\phi_0^2\,\mathfrak{m})}\cdot \left\|g\right\|_{L^1(E;\phi_0^2\,\mathfrak{m})}
$.
Using this fact with \eqref{FKFET}, \eqref{vFt_cvg} and \eqref{FKETLpq}, we deduce
\begin{equation}\label{LB_fg}
\liminf_{t\to \infty}\varPhi_t(x)^{-1}\,
\widehat p^{\mu, F}_{pt}\left(f\cdot \widehat p_{(q-p)t}^{\mu, F} \left(g\cdot \varPhi_{(1-q)t}\right)\right)(x)
\ge \frac{\int_E fh_{r_0}^{(\rho)}\phi_0^2\d\mathfrak{m}}
{\int_E \phi_0\d\mathfrak{m}}.
\end{equation}
Similarly, we see $g h_{r_0}^{(\mathbf{1})}\in L^1(E;\phi^2_0\mathfrak{m})$,
so that thanks to \eqref{FKFET} and to the monotone convergence theorem:
\begin{equation}\label{hrho}
\begin{split}
	\liminf_{r_0\to \infty} h_{r_0}^{(\rho)}(y)&\ge \liminf_{r_0\to \infty}\left(\lim_{r\to \infty}p_r^{\phi_0}\left(g h_{r_0}^{(\mathbf{1})}\right)(y)\right) \\
	&\ge \liminf_{r_0\to \infty} \int_E g h_{r_0}^{(\mathbf{1})}\phi^2_0\d\mathfrak{m} \ge \int_E \phi_0 \d\mathfrak{m}\cdot \int_E g\phi_0^2 \d\mathfrak{m}.
\end{split}
\end{equation}
Injecting this lower-bound in \eqref{LB_fg} and using the monotone convergence theorem, we obtain
\begin{equation*}
\liminf_{t\to \infty}\varPhi_t(x)^{-1}\,
\widehat p^{\mu, F}_{pt}\left(f\cdot \widehat p_{(q-p)t}^{\mu, F} \left(g\cdot \varPhi_{(1-q)t}\right)\right)(x) 
\ge \int_E f\phi_0^2\d\mathfrak{m} \cdot \int_E g\phi_0^2\d\mathfrak{m}.
\end{equation*}

For the converse result, note that the convergence is already known for $g \equiv 1$ with a semigroup property, as a direct application of $(1)$.
The upper-bound on the supremum limit is thus deduced 
by considering the infimum limit of $\varPhi_t(x)^{-1}\,
\widehat p^{\mu, F}_{pt}(f\cdot \widehat p_{(q-p)t}^{\mu, F} ((n-g_n)\cdot \varPhi_{(1-q)t}))(x)$ 
with a very similar reasoning as for $(1)$, where $g_n = \min \{g, n\}$. 
Therefore,
\begin{equation}\label{eq_ext_L1}
\limsup_{t\to \infty}\varPhi_t(x)^{-1}\,
\widehat p^{\mu, F}_{pt}\left(f\cdot \widehat p_{(q-p)t}^{\mu, F} \left( g_n\cdot \varPhi_{(1-q)t}\right)\right)(x) 
\le \int_E f\phi_0^2\d\mathfrak{m} \cdot \int_E g\phi_0^2\d\mathfrak{m}.
\end{equation}
As for \eqref{eq_L1_approx},
\begin{equation}\label{eq_L1_approx2}
	\begin{split}
		\widehat p^{\mu, F}_{pt}\left(f\cdot \widehat p_{(q-p)t}^{\mu, F} \left((g-g_n)\cdot \varPhi_{(1-q)t}\right)\right)(x) 
		&\le c_{r_0}\phi_0(x) \left\|p_{r_0}^{\phi_0}f\right\|_{L^\infty(E;\phi_0^2\mathfrak{m})} \int_E  (g-g_n)\cdot p_{(1-q)t}^{\phi_0} \left(\frac{1}{\phi_0}\right) \phi_0^2\d\mathfrak{m} \\
		&\le c_{r_0}\phi_0(x) \left\|p_{r_0}^{\phi_0}f\right\|_{L^\infty(E;\phi_0^2\mathfrak{m})}
		\left\|p_{r_0}^{\phi_0}\left(g - g_n\right)\right\|_{L^\infty(E;\phi_0^2\mathfrak{m})}
		\int_E  \phi_0\d\mathfrak{m} \\
		&\le c_{r_0}^3\phi_0(x) \left\|f\right\|_{L^1(E;\phi_0^2\mathfrak{m})}
		\left\|g-g_n\right\|_{L^1(E;\phi_0^2\mathfrak{m})}\int_E  \phi_0\d\mathfrak{m}. 
	\end{split}
\end{equation}	
By the monotone convergence theorem, the $L^1(E;\phi_0^2\mathfrak{m})$-norm of $g-g_n$ tends to 0 as $n\to \infty$.
With \eqref{eq_L1_approx2}, we can pass to the limit $n\to \infty$ 
in \eqref{eq_ext_L1} and deduce the desired result for any nonnegative $f$ and $g$. It extends for arbitrary $f$ and $g$ by linearity. 
\vskip 0.1cm 

(3): The proof of the assertion follows the same lines as the proof of (1), with the additional symmetry property 
stated in \eqref{eq_sym_Jphi}.
First, we suppose  that $G$ and $g$ are nonnegative.
The $L^1$-norm according to the measure $\phi_0\mathfrak{m}$
of the function $N[G e^{-F} \phi_0 h_{r_0}^{(g)}]$
is upper-bounded thanks to the symmetry property:
\begin{equation}\label{L1NFg}
\begin{split}
	\int_E N\left[G e^{-F} \phi_0 h_{r_0}^{(g)}\right] \phi_0 \d \mathfrak{m} 
	&=\int_E N^{\phi_0}\left[G\cdot h_{r_0}^{(g)}\right] \phi_0^2 \d \mathfrak{m} 
	= \int_E  h_{r_0}^{(g)}\, N^{\phi_0}\left[G^T\right]\, \phi_0^2 \d \mathfrak{m} 
	\\
	&\le \int_E p^{\phi_0}_{r_0}\left(\frac{g}{\phi_0}\right) \,N^{\phi_0}\left[G^T\right] \phi_0^2 \d \mathfrak{m} 
	\le  c_{r_0}\left\|N^{\phi_0}\left[G^T\right]\right\|_{L^1(E; \phi_0^2\,\mathfrak{m})}
	\cdot \int_E g \phi_0\,\d\mathfrak{m} <\infty.
\end{split}
\end{equation}
In the right-hand side above, we used that $\{p_t^{\phi_0}\}_{t\ge 0}$ satisfies (UC), together with $N^{\phi_0}[G^T] \in L^1(E;\phi_0^2\mathfrak{m})$ due to \eqref{eq_sym_Jphi}. This implies the following limit 
thanks to \eqref{FKFET} and to \eqref{vFt_cvg},
since $(1-p) t \ge r_0$ hold for $t$ sufficiently large:
\begin{align*}
\liminf_{t\to \infty}\varPhi_t(x)^{-1}\,
\widehat p^{\mu, F}_{pt}\left(N\left[G e^{-F}\cdot \widehat p^{\mu, F}_{(1-p)t}g\right]\right)(x)
&\ge \liminf_{t\rightarrow \infty}\varPhi_t(x)^{-1}\,
\widehat p^{\mu, F}_{pt}\left(N\left[G e^{-F} \phi_0 h_{r_0}^{(g)}\right]\right)(x) 
\\&= \dfrac{\int_E N[G e^{-F} \phi_0 h_{r_0}^{(g)}] \phi_0 \d \mathfrak{m}}
{\int_E \phi_0\d\mathfrak{m}}.
\end{align*}
Letting $r_0 \to \infty$ (hence $h_{r_0}^{(g)}(\cdot) \to \int_E g\phi_0\d\mathfrak{m}$), again by the monotone convergence theorem, 
we conclude the lower-bound:
\begin{align*}
\liminf_{t\to \infty}\varPhi_t(x)^{-1}\,\widehat p^{\mu, F}_{pt}\left(N\left[G e^{-F}\cdot \widehat p^{\mu, F}_{(1-p)t}g\right]\right)(x)
&\ge \int_E N^{\phi_0}[G] \phi^2_0 \d\mathfrak{m} \cdot \int_E g \d\nu_{\phi_0} \\
&=\iint_{E\times E}G(y,z){\mathcal J}_{\phi_0}({\d}y{\d}z)\cdot \int_E g \d\nu_{\phi_0}.
\end{align*}

As for $(1)$, because the function $N[G e^{-F}\cdot\widehat p^{\mu, F}_{(1-p)t}\phi_0]
= N^{\phi_0}[G]\phi_0 \in L^1(E; \phi_0 \mathfrak{m})$
does not depend on $t$, we  notice that the limit is already known 
provided $g = \phi_0$.
So we define $g_n = \min \{g, n\phi_0\}$ for any $n\ge 1$ and identify by using linearity the relevant lower- and upper-bound of 
\begin{equation*}
\liminf_{t\rightarrow \infty}\varPhi_t(x)^{-1}
\widehat p^{\mu, F}_{pt}\left(N\left[G e^{-F}\cdot \widehat p^{\mu, F}_{(1-p)t}(n\phi_0 - g_n)\right]\right)(x)
\end{equation*}
to deduce the following:
\begin{equation*}
\limsup_{t\rightarrow \infty}\varPhi_t(x)^{-1}\,\widehat p^{\mu, F}_{pt}\left(N\left[G e^{-F}\cdot \widehat p^{\mu, F}_{(1-p)t}g_n\right]\right)(x) \le 
\iint_{E\times E}G(y,z){\mathcal J}_{\phi_0}({\d}y{\d}z)\cdot \int_E g_n \d\nu_{\phi_0}.
\end{equation*}
This concludes the upper-bound in $(3)$ for $g_n$ and any nonnegative $G$.
 As for $(1)$, with an adaptation of \eqref{L1NFg} instead of \eqref{eq_L1_control},
we can pass to the limit $n\rightarrow \infty$
and derive the upper-bound in $(3)$ for any nonnegative $g$ and $G$.
It extends to signed functions $g$ and $G$ by linearity.
\vskip 0.1cm

(4): The procedure so far can be iterated in the sense of the assertion $(4)$, as for the assertion $(2)$ without specific difficulty. Suppose first that $g$, $G$ and $K$ are nonnegative. For each $r_0 >0$ such that both $(1-q)t \ge r_0$ and $(q-p)t \ge r_0$, 
the inequality below is deduced by virtue of \eqref{MIN_PHI}:
\begin{equation*}
\varPhi_t(x)^{-1}\,\widehat p^{\mu, F}_{pt}\Big(N\Big[G e^{-F} \cdot \widehat p^{\mu, F}_{(q-p)t}\left(N\left[K e^{-F} \cdot \widehat p^{\mu, F}_{(1-q)t}g\right]\right)\Big]\Big)(x) \ge  \varPhi_t(x)^{-1}\,\widehat p^{\mu, F}_{pt}\left(N\left[G e^{-F} \cdot\phi_0 h_{r_0}^{(\rho)}\right]\right)(x),
\end{equation*}
where $\rho(x) = N[K e^{-F} \phi_0 h_{r_0}^{(g)}](x)$.
By the definition of $h_{r_0}^{(\rho)}$ and the fact that $\{p_t^{\phi_0}\}_{t\ge 0}$ satisfies (UC), 
one can also check as in \eqref{L1NFg} that 
$N[G e^{-F}\phi_0 h_{r_0}^{(\rho)}]\in L^1(E;\phi_0\mathfrak{m})$,
with a norm upper-bounded by
\begin{equation*}
c_{r_0}^2 \left\|N^{\phi_0}[G^T]\right\|_{L^1(E;\phi_0^2\mathfrak{m})}\cdot \left\|N^{\phi_0}[K^T]\right\|_{L^1(E;\phi_0^2\mathfrak{m})}
\cdot\int_E g \phi_0 {\rm d}\mathfrak{m}< \infty.
\end{equation*} 
Thus, by \eqref{FKFET}, \eqref{vFt_cvg} and \eqref{eq_sym_Jphi},
\begin{equation*}
\begin{split}
	&	\liminf_{t\rightarrow \infty}
	\varPhi_t(x)^{-1}\, 
	\widehat p^{\mu, F}_{pt}\Big(N\Big[G e^{-F}
	\cdot	\widehat p^{\mu, F}_{(q-p)t}\left(N\left[K e^{-F} 
	\cdot\widehat p^{\mu, F}_{(1-q)t}g\right]\right)\Big]\Big)(x) \\
	&\qquad \quad \ge \dfrac{\int_E N[G e^{-F} \phi_0 h_{r_0}^{(\rho)}] \phi_0 \d \mathfrak{m}}
	{\int_E \phi_0\d\mathfrak{m}} = \dfrac{\int_E N^{\phi_0}[G\cdot h_{r_0}^{(\rho)}] \phi_0^2 \d\mathfrak{m}}
	{\int_E \phi_0\d\mathfrak{m}}=\dfrac{\int_E N^{\phi_0}[G]h_{r_0}^{(\rho)}\phi_0^2\d \mathfrak{m}}{\int_E \phi_0\d \mathfrak{m}}.
\end{split}
\end{equation*}
We may then exploit \eqref{hrho} to deduce the following: for any $y \in E$,
\begin{equation*}
\begin{split}
	\liminf_{r_0\rightarrow \infty}h_{r_0}^{(\rho)}(y) 
	\ge \liminf_{r_0\rightarrow \infty} \int_E N\left[K e^{-F}\phi_0 h_{r_0}^{(g)}\right] \phi_0 \d\mathfrak{m}
	= \int_E N^{\phi_0}[K] \phi^2_0 \d\mathfrak{m}\cdot \int_E g \phi_0 \d\mathfrak{m}.
\end{split}
\end{equation*}
With these two estimates, we conclude the lower-bound:
\begin{equation*}
\begin{split}
	&\liminf_{t\rightarrow \infty}\varPhi_t(x)^{-1}\,\widehat p^{\mu, F}_{pt}\Big(N\Big[Ge^{-F}\cdot\widehat p^{\mu, F}_{(q-p)t}\left(N\left[K e^{-F}\cdot\widehat p^{\mu, F}_{(1-q)t}g\right]\right)\Big]\Big)(x)\\
	&\qquad \quad \ge \int_E N^{\phi_0}[G] \phi_0^2 \d\mathfrak{m}\cdot \int_E N^{\phi_0}[K] \phi^2_0 \d\mathfrak{m}\cdot \int_E g \d\nu_{\phi_0} \\
	&\qquad \quad = \iint_{E\times E}G(y,z){\mathcal J}_{\phi_0}({\d}y{\d}z) \cdot \iint_{E\times E}K(y,z){\mathcal J}_{\phi_0}({\d}y{\d}z) \cdot \int_E g {\rm d}\nu_{\phi_0}.
\end{split}
\end{equation*}

For the converse inequality, we again focus on the case where $g = \phi_0$, in which case the left-hand side of the above inequality takes the following form:
\begin{equation*}
	\liminf_{t\to \infty}\varPhi_t(x)^{-1}\widehat{p}_{pt}^{\mu,F}\Big(N\Big[Ge^{-F}\cdot \widehat{p}_{(q-p)t}^{\mu,F}\left(N\left[Ke^{-F}\phi_0\right]\right)\Big]\Big)(x). 
\end{equation*} 
On the other hand, the assertion $(3)$ implies the following convergence, with $g = N[Ke^{-F}\phi_0]=N^{\phi_0}[K]\phi_0$,
$t' = qt$ and $p' = p/q \in (0, 1)$:
\begin{equation*}
	\begin{split}
		&\lim_{t\rightarrow \infty}\varPhi_{qt}(x)^{-1}\widehat{p}_{pt}^{\mu,F}\Big(N\Big[Ge^{-F}\cdot \widehat{p}_{(q-p)t}^{\mu,F}\left(N\left[K e^{-F}\phi_0\right]\right)\Big]\Big)(x)\\
		&\qquad = \lim_{t'\rightarrow \infty}\varPhi_{t'}(x)^{-1}\,\widehat p^{\mu, F}_{p't'}\Big(N\Big[Ge^{-F}\cdot	\widehat p^{\mu, F}_{(1-p')t'}\left(N\left[K e^{-F}\phi_0\right]\right)\Big]\Big)(x) \\
		&\qquad = \iint_{E\times E}G(y,z){\mathcal J}_{\phi_0}({\d}y{\d}z)\cdot \int_E N^{\phi_0}[K]\phi_0 \d\nu_{\phi_0} \\
		&\qquad = \iint_{E\times E}G(y,z){\mathcal J}_{\phi_0}({\d}y{\d}z) \cdot \iint_{E\times E} K(y,z){\mathcal J}_{\phi_0}({\d}y{\d}z) \cdot \int_E \phi_0 {\rm d}\nu_{\phi_0}.
	\end{split}
\end{equation*}
Thanks to \eqref{vFt_cvg}, $\varPhi_{qt}/\varPhi_{t}$ tends to 1 as $t\to \infty$, which concludes the proof of $(4)$ in the case where $g=\phi_0$. From this fact, one can deduce the upper-bound on the supremum limit 
 by the same approximation scheme as in
the proof of~$(3)$.
The reasoning in \eqref{eq_L1_approx2}
can similarly be adapted with 
\eqref{L1NFg} instead of \eqref{eq_L1_control}
to deal with the error term  associated with replacing $g$ by the nondecreasing sequence $g_n$.
The result finally extends to signed functions $G$ and $K$ by linearity.

\end{proof}

\subsection{Quasi-ergodic limits for additive functionals}

We now give a proof of the following quasi-ergodic limit theorem for pairs of continuous and discontinuous additive functionals under the Feynman--Kac transforms.

\begin{thm}\label{JumpErgodic}
For any $V \in L^1(E;\phi_0^2\mathfrak{m})$ and any $G\in \mathcal K(E\times E)$ such that $N^{\phi_0}[|G|] \in L^1(E;\phi^2_0\mathfrak{m})$, suppose that for some $r>0$, $\varepsilon>0$ and for any $x\in E$
\begin{equation}\label{eq_bd_AfG}
		\sup_{t\ge r} {\bE}_{x|t}^{\mu,F}\left[\left|A_\varepsilon^{V,G}\right|\right] <\infty,
		\quad \sup_{t\ge r} {\bE}_{x|t}^{\mu,F}\left[\left|A_t^{V,G}-A_{t-\varepsilon}^{V,G}\right|\right] <\infty.
\end{equation}
Then, we have for any $x \in E$
\begin{equation*}
\lim_{t\to \infty}{\mathbb{E}}_{x|t}^{\mu, F}\left[\frac{1}{t}A_t^{V,G}\right] = \int_E V\phi_0^2 \d\mathfrak{m} + \iint_{E\times E}G(x,y){\mathcal J}_{\phi_0}({\d}x{\d}y).
\end{equation*}
\end{thm}

\begin{proof}
The proof is a consequence of the results in the previous subsections. First, we suppose that $V$ and $G$ are nonnegative. 
The equality below follows from \eqref{V_dist} and \eqref{tg}:
\begin{align*}
	\lim_{t\to \infty}{\mathbb{E}}_{x|t}^{\mu, F}\left[\frac{1}{t}A_t^{V,G}\right] 
	&=\lim_{t\to \infty}\frac{1}{t}\int_0^t \varPhi_t(x)^{-1} \widehat p^{\mu, F}_{s}\Big(V\cdot \varPhi_{t-s} + N\left[Ge^{-F}\cdot \varPhi_{t-s}\right]\Big)(x)\d s \\ 
	&=\lim_{t\to \infty}\int_0^1 \varPhi_t(x)^{-1}\widehat p^{\mu, F}_{pt}\Big(V\cdot \varPhi_{(1-p)t} +N\left[Ge^{-F}\cdot \varPhi_{(1-p)t}\right]\Big)(x)\d p
\end{align*}
Note that the semigroup in the integral on the right-hand side above does not guarantee its boundedness when $p$ is sufficiently close to 0 or 1, under only the current integrability conditions of $V$ and $N^{\phi_0}[G]$ with respect to $\phi_0^2\mathfrak{m}$. Let us take $t>0$ large enough such that $t\ge \max \{\varepsilon /p, \varepsilon /(1-p)\}$ and $\varPhi_t(x) \ge (\phi_0(x)/2)\int_E\phi_{0} {\d}\mathfrak{m}$ by recalling \eqref{FKFET}. Then, we see by (UC) of $(p_t^{\phi_0})_{t\ge 0}$ together with \eqref{eq_phi_ratio} and \eqref{eq_sym_Jphi} that $\varPhi_t^{-1}\widehat p^{\mu, F}_{pt}(V\cdot \varPhi_{(1-p)t} +N[Ge^{-F}\cdot \varPhi_{(1-p)t}])$ is uniformly upper-bounded for such a large enough $t>0$: 
\begin{align}\label{eq_L1_bound}
	\begin{split}
		&\left\|\varPhi_t^{-1}\widehat p^{\mu, F}_{pt}\Big(V\cdot \varPhi_{(1-p)t} +N\left[Ge^{-F}\cdot \varPhi_{(1-p)t}\right]\Big)\right\|_{L^\infty (E;\phi_0^2\mathfrak{m})} \\
		&\quad \le \frac{2}{\int_E \phi_0 {\d}\mathfrak{m}}\left\|p_{pt}^{\phi_0}\Big(V\cdot p_{(1-p)t}^{\phi_0}\left({1/\phi_0}\right)+N^{\phi_0}\left[G\cdot p_{(1-p)t}^{\phi_0}\left({1/\phi_0}\right)\right]\Big)\right\|_{L^\infty (E;\phi_0^2\mathfrak{m})} \\
		&\quad \le \frac{2c_\varepsilon}{\int_E \phi_0 {\d}\mathfrak{m}}\left\|V\cdot p_{(1-p)t}^{\phi_0}\left({1/\phi_0}\right)+N^{\phi_0}\left[G\cdot p_{(1-p)t}^{\phi_0}\left({1/\phi_0}\right)\right]\right\|_{L^1 (E;\phi_0^2\mathfrak{m})} \\
		&\quad \le \frac{2c_\varepsilon}{\int_E \phi_0 {\d}\mathfrak{m}}\left\|p_{(1-p)t}^{\phi_0}({1/\phi_0})\right\|_{L^\infty (E;\phi_0^2\mathfrak{m})} \left(\|V\|_{L^1(E;\phi_0^2\mathfrak{m})}+\left\|N^{\phi_0}\left[G^T\right]\right\|_{L^1 (E;\phi_0^2\mathfrak{m})}\right) \\
		&\quad \le 2c_\varepsilon^2 \left(\|V\|_{L^1(E;\phi_0^2\mathfrak{m})}+\left\|N^{\phi_0}\left[G^T\right]\right\|_{L^1 (E;\phi_0^2\mathfrak{m})}\right) <\infty.
	\end{split}
\end{align}
By combining \eqref{eq_L1_bound} with the following decomposition adapted from \eqref{V_dist} and \eqref{tg},
\begin{equation}
	{\mathbb{E}}_{x|t}^{\mu, F}\left[A_r^{V,G}\right] =\int_0^r \varPhi_t(x)^{-1}\widehat p^{\mu, F}_{s}\Big(V\cdot \varPhi_{t-s} + N\left[Ge^{-F}\cdot \varPhi_{t-s}\right]\Big)(x)\d s,
\end{equation}
which holds for any $r\in [0, t]$, the following result is justified by the dominated convergence theorem: 
\begin{align*}
	\lim_{t\to \infty}{\mathbb{E}}_{x|t}^{\mu, F}\left[\frac{1}{t}\left(A_{t-\varepsilon}^{V,G}-A_\varepsilon^{V,G}\right)\right] 
	&\quad =\lim_{t\to \infty}\int_{\varepsilon/t}^{1-\varepsilon/t} \varPhi_t(x)^{-1}\widehat p^{\mu, F}_{pt}\Big(V\cdot \varPhi_{(1-p)t} + N\left[Ge^{-F}\cdot \varPhi_{(1-p)t}\right]\Big)(x)\d p \\
	&\quad =\int_E V\phi_0^2 \d\mathfrak{m} + \iint_{E\times E}G(x,y){\mathcal J}_{\phi_0}({\d}x{\d}y).
\end{align*}
The result, justified for nonnegative functions $V, G$, extends by linearity to signed functions.

On the other hand, by \eqref{eq_bd_AfG}, 
\begin{align*}
	&\left|{\bE}_{x|t}^{\mu,F}\left[\frac{1}{t}A_t^{V,G}\right] -{\bE}_{x|t}^{\mu,F}\left[\frac{1}{t}\left(A_{t-\varepsilon}^{V,G}-A_{\varepsilon}^{V,G}\right)\right]\right| \\
	&\qquad \le \frac{1}{t}{\bE}_{x|t}^{\mu,F}\left[\left|A_t^{V,G}-A_{t-\varepsilon}^{V,G}\right|\right]
	+\frac{1}{t}{\bE}_{x|t}^{\mu,F}\left[\left|A_\varepsilon^{V,G}\right|\right] ~~\longrightarrow ~~0 \quad \text{as}~~t\to \infty\,.
\end{align*}
This concludes the proof of Theorem~\ref{JumpErgodic}.
\end{proof}

Theorem \ref{JumpErgodic} tells us that the quasi-ergodic limits for both the occupation time and the number of jumps 
caused by $X$ under ${\mathbb{P}}_{x|t}^{\mu, F}$ are characterized by the measures $\phi_0^2\mathfrak{m}$ and ${\mathcal J}_{\phi_0}$, respectively. In particular, $\phi_0^2\mathfrak{m}$ is a unique quasi-ergodic distribution of $X$ under ${\mathbb{P}}_{x|t}^{\mu, F}$.
\vskip 0.1cm
The following lemma shows that we can replace the condition \eqref{eq_bd_AfG} in Theorem~\ref{JumpErgodic} by \eqref{eq_bd_AfGphi0} below, after possibly replacing both $V$ and $G$ by their absolute values:

\begin{lem}\label{lem_AF}
For any nonnegative measurable function $V$ on $E$ and any nonnegative $G \in {\mathcal K}(E\times E)$, suppose that for some $\varepsilon >0$ and for any $x \in E$:
\begin{equation}\label{eq_bd_AfGphi0}
	\mathbb E_x^{\phi_0}\left[\left(A^{V, G}_\varepsilon\right)^2\right] <\infty,
	\quad \int_E {\mathbb E}_y^{\phi_0}\left[\left(A^{V, G^T}_{\varepsilon}\right)^2\right]\phi_0(y)\mathfrak{m}(\d y) <\infty.
\end{equation}
Then, the condition \eqref{eq_bd_AfG} holds.
\end{lem}
\begin{proof}

First, note that the finiteness conditions with respect to $(A_\varepsilon^{V,G})^2$ and $(A_\varepsilon^{V,G^T})^2$ in \eqref{eq_bd_AfGphi0} imply finiteness with respect to $ A_\varepsilon^{V,G}$ and $A_\varepsilon^{V,G^T}$ 
due to the Cauchy-Schwarz inequality. Recalling \eqref{vFt_cvg}, 
let $r\ge 2$ be sufficiently large so that:
\begin{equation}\label{eq_vFt_min}
	\inf_{t\ge r} \varPhi_t(x)\ge \frac{\phi_0(x)}2\, \int_E \phi_0{\d}\mathfrak{m} \,.
\end{equation}
By the Markov property at time $\varepsilon \in (0,1)$,
\begin{align*}
	\mathbb E_{x}\left[A^{V, G}_{\varepsilon}e^{\lambda_0 t}  W^{\mu, F}_t\right] &= \mathbb E_{x}\left[A^{V, G}_{\varepsilon}e^{\lambda_0 \varepsilon}  W^{\mu, F}_\varepsilon \varPhi_{t-\varepsilon}(X_\varepsilon)\right] \\
	&\le c'_{1} \mathbb E_{x}\left[A^{V, G}_{\varepsilon}e^{\lambda_0 \varepsilon}  W^{\mu, F}_\varepsilon \phi_0(X_\varepsilon)\right] \\
	&= c'_{1}  \phi_0(x) \cdot \mathbb E^{\phi_0}_{x}\left[A^{V, G}_{\varepsilon}\right],
\end{align*}
where we recall \eqref{eq_phi_ratio} 
with $c'_1 = c_1 \int_E \phi_0 {\rm d}\mathfrak m$. Together with  \eqref{eq_vFt_min}, we thus deduce for any $t\ge r$:
\begin{equation*}
	\mathbb E_{x|t}^{\mu, F}\left[A^{V, G}_{\varepsilon}\right] \le 2c_1 \mathbb E^{\phi_0}_{x}\left[A^{V, G}_{\varepsilon}\right] <\infty.
\end{equation*}
On the other hand, by the Markov property at time $t-\varepsilon$,
\begin{equation}\label{eq_fin_AF}	
	\mathbb E_{x}\left[(A^{V, G}_t - A^{V, G}_{t-\varepsilon}) e^{\lambda_0 t}  W^{\mu, F}_t\right]	
	= \phi_0(x) p_{t-\varepsilon}^{\phi_0}\left(\xi_\varepsilon^{V, G}\right)(x),
\end{equation}
where $\xi_\varepsilon^{V, G}(x):= \phi_0(x)^{-1} \mathbb E_{x}[A^{V, G}_{\varepsilon} e^{\lambda_0 \varepsilon}\, W^{\mu,F}_{\varepsilon}]$. A straightforward adaptation of the justification behind \eqref{V_dist} 
and \eqref{tg} leads to:
\begin{equation}\label{eq_A_Xi}
	\begin{split}
		\xi_\varepsilon^{V, G}(x)
		&= \phi_0(x)^{-1}  \int_0^\varepsilon \widehat{p}_{s}^{\mu,F}\left(V\cdot \varPhi_{\varepsilon-s}+N\left[Ge^{-F}\cdot \varPhi_{\varepsilon-s}\right]\right)(x)\d s 
		\\&=  \int_0^\varepsilon p_{s}^{\phi_0}\left(V\cdot p^{\phi_0}_{\varepsilon-s}(1/\phi_0)+N^{\phi_0}\left[G\cdot p^{\phi_0}_{\varepsilon-s}(1/\phi_0)\right]\right)(x)\d s\,.
	\end{split}
\end{equation}
Similarly,
\begin{equation}\label{eq_A_phi}
	{\mathbb E}_y^{\phi_0}\left[A_{\varepsilon}^{V,G^T}\right]
	= \int_0^\varepsilon p^{\phi_0}_{\varepsilon-s} \left(V + N^{\phi_0}\left[G^T\right]\right)(x)\d s.
\end{equation}
\eqref{eq_A_Xi} and \eqref{eq_A_phi} together with the symmetry of $(p_t^{\phi_0})_{t\ge 0}$ 
entail that $\xi_\varepsilon^{V, G}$ belongs to $L^1(E;\phi_0^2\mathfrak{m})$. Indeed,
\begin{equation*}
	\begin{split}
		\int_E \xi_\varepsilon^{V, G} \phi_0^2 {\rm d}\mathfrak m 
		&=\int_0^\varepsilon \int_E p_{ s}^{\phi_0}\left(V\cdot p^{\phi_0}_{\varepsilon- s}(1/\phi_0)+N^{\phi_0}\left[G\cdot p^{\phi_0}_{\varepsilon- s}(1/\phi_0)\right]\right)(y) \phi_0^2(y)\mathfrak{m}(\d y) \d s 
		\\&=\int_0^\varepsilon \int_E \left(V\cdot p^{\phi_0}_{\varepsilon- s}(1/\phi_0)+N^{\phi_0}\left[G^T\right]\cdot p^{\phi_0}_{\varepsilon- s}(1/\phi_0)\right)(y) \phi_0^2(y)\mathfrak{m}(\d y) \d s \\
		&=\int_0^\varepsilon \int_E p^{\phi_0}_{\varepsilon- s} \left(V + N^{\phi_0}\left[G^T\right]\right)(y)\, \phi_0(y)\mathfrak{m}(\d y)\, \d s
		\\	&=\int_E {\mathbb E}_y^{\phi_0}\left[A_{\varepsilon}^{V,G^T}\right]\phi_0(y)\mathfrak{m}(\d y) <\infty.
	\end{split}
\end{equation*}
Then, by recalling \eqref{eq_vFt_min} and \eqref{eq_fin_AF}, we deduce for any $t\ge r$:
\begin{equation*}
	\mathbb E_{x|t}^{\mu, F}\left[A^{V, G}_{t} - A^{V, G}_{t-\varepsilon}\right] \le \frac{2}{\int_E \phi_0 \d\mathfrak{m}} \int_E {\mathbb E}_y^{\phi_0}\left[A_{\varepsilon}^{V,G^T}\right]\phi_0(y)\mathfrak{m}(\d y) <\infty.
\end{equation*}
The proof is complete. 
\end{proof}
\vskip 0.1cm
We say that a signed smooth measure $\nu:=\nu^+-\nu^-$ on $E$ in the strict sense belongs to the Dynkin class associated to $X$ ($\nu \in {\mathcal S}_D^1(X)$ in notation), if it satisfies that
\begin{equation*}
\sup_{x \in E}\mathbb{E}_x\left[|A_t^\nu|\right] <\infty
\end{equation*} 
for some, hence for any $t>0$. 
It is clear that ${\mathcal S}_K^1(X) \subset {\mathcal S}_D^1(X)$. In what follows, for any measurable function $V$ from $E$ to $\mathbb R$, $V\in {\mathcal S}_D^1(X)$ is to be understood as $V\mathfrak m\in {\mathcal S}_D^1(X)$, and similarly for ${\mathcal S}_K^1(X)$ instead of ${\mathcal S}_D^1(X)$.

\begin{lem}\label{lem_AF_id}
Let $V$ and $G\in \mathcal K(E\times E)$ be such that $V, N[|G|], N[|G|^T] \in {\mathcal S}_D^1(X)$. Then, the condition \eqref{eq_bd_AfGphi0} holds with $|V|, |G|$ instead of $V, G$, thus also \eqref{eq_bd_AfG}.
\end{lem}
\begin{proof}
Given Lemma~\ref{lem_AF}, it suffices to prove the condition \eqref{eq_bd_AfGphi0}.
By the Cauchy-Schwarz inequality,
\begin{align}\label{finitenessAVG2}
	\begin{split}
		\mathbb E_x^{\phi_0}\left[\left|A^{V, G}_\varepsilon\right|^2\right]&\le e^{\lambda_0 \varepsilon}\frac{\|\phi_0\|_\infty}{\phi_0(x)}{\mathbb E}_x\left[\exp \left(\left|A_\varepsilon^{\mu,F}\right|\right)\left|A^{V, G}_\varepsilon\right|^2\right] \\
		&\le e^{\lambda_0 \varepsilon}\frac{\|\phi_0\|_\infty}{\phi_0(x)}{\mathbb E}_x\left[\exp \left(2\left|A_\varepsilon^{\mu,F}\right|\right)\right]^{1/2}{\mathbb E}_x\left[\left|A^{V, G}_\varepsilon\right|^4\right]^{1/2}.  
	\end{split}
\end{align}
Put $\mu_2:=2|\mu|$ and $F_2:=e^{2|F|}-1$. By the boundedness of $F$, it is easy to check that the measures $\mu_2$ and $N[F_2]$ are of Kato class associated to $X$.
Let ${\rm Exp} (A)_t:=e^{A_t^c}\prod_{0<s\le t}(1+\Delta A_s)$ be the Stieltjes exponential of a positive additive functional $A_t$, where $A_t^c$ denotes the continuous part of $A_t$ and $\Delta A_s=A_s-A_{s-}$ (see [32, (2.5)]). Then  
\begin{equation*}
	\exp \left(2\left|A_\varepsilon^{\mu,F}\right|\right) \le\, \exp \left(A_\varepsilon^{2|\mu|,2|F|}|\right)={\rm Exp} \left(A^{\mu_2,F_2}\right)_\varepsilon. 
\end{equation*}
From this, with Khas'minskii's lemma for a positive additive functional ([32, Lemma 2.1]), we see
\begin{equation}\label{Khas}
	{\mathbb E}_x\left[\exp \left(2\left|A_\varepsilon^{\mu,F}\right|\right)\right]\le\, {\mathbb E}_x\left[{\rm Exp} \left(A^{\mu_2,F_2}\right)_\varepsilon\right] \le \frac{1}{1-\sup_{x\in E}{\mathbb E}_x[A^{\mu_2,F_2}_\varepsilon]} <\infty\,.
\end{equation}
for sufficiently small $\varepsilon>0$. Further, by virtue of [17, Lemma 2.1], one has
\begin{equation}\label{interpol}
	{\mathbb E}_x\left[\left|A^{V, G}_\varepsilon\right|^4\right] \le 24\left(\sup_{x\in E}{\mathbb E}_x\left[\left|A^{V, G}_\varepsilon\right|\right]\right)^4 <\infty.
\end{equation} 
Now, the finiteness of the left-hand side of \eqref{finitenessAVG2} easily follows from \eqref{Khas} and \eqref{interpol}. This procedure also applies to showing the second condition in \eqref{eq_bd_AfGphi0}.
\end{proof}

We observe the following inequality due to \eqref{def_Nphi} and the boundedness of both $\phi_0$ and $F$:
\begin{equation}\label{eq_prop_L1phi}
\left\|N^{\phi_0}[|G|]\right\|_{L^1(E; \phi_0^2 \mathfrak m)}
= \int_E N\left[|G|\, e^{-F}\, \phi_0\right] \phi_0 {\d}\mathfrak m \le 	\|\phi_0\|_{L^\infty(E; \mathfrak m)}^2
\, \left\|e^{-F}\right\|_{L^\infty(E; \mathfrak m)}
\, \|N[|G|]\,\|_{L^1(E; \mathfrak m)}\,.
\end{equation}
So $N^{\phi_0}[|G|] \in L^1(E; \phi_0^2 m)$ is implied by $N[|G|] \in L^1(E; m)$, and similarly $L^1(E; m)\subset L^1(E; \phi_0^2 m)$.
\vskip 0.1cm
In view of Lemma \ref{lem_AF} and Lemma \ref{lem_AF_id}, we have the following corollary, which plays a role in Section 4.

\begin{cor}\label{JumpErgodicCor}
For any $V$ and $G\in \mathcal K(E\times E)$ such that
$V, N[|G|], N[|G|^T] \in L^1(E;\mathfrak{m})\cap {\mathcal S}_D^1(X)$,
we have for any $x \in E$:
\begin{equation*}
	\lim_{t\to \infty}{\mathbb{E}}_{x|t}^{\mu, F}\left[\frac{1}{t}A_t^{V,G}\right] = \int_E V\phi_0^2 \d\mathfrak{m} + \iint_{E\times E}G(x,y){\mathcal J}_{\phi_0}({\d}x{\d}y).
\end{equation*}
In particular, for any $x \in E$
\begin{equation*}
	\lim_{t\to \infty}{\mathbb{E}}_{x|t}^{V, F}\left[\frac{1}{t}A_t^{V,F}\right] = \int_E V\phi_0^2 \d\mathfrak{m} + \iint_{E\times E}F(x,y){\mathcal J}_{\phi_0}({\d}x{\d}y)
\end{equation*}
provided $V$, $N[|F|]  \in L^1(E;\mathfrak{m}) \cap {\mathcal S}_K^1(X)$.
\end{cor}

\subsection{Conditional functional weak law of large numbers}\label{sec:WLLN}  

In this subsection, we establish conditional functional weak laws of large numbers for $X$ under ${\mathbb{P}}_{x|t}^{\mu, F}$. The result will play an important role in the next section. 
We first present the two Lemmas~\ref{Moment2}-\ref{JumpWLLN} that extend Theorem~\ref{JumpErgodic}
for the case of second moments. The approach is similar yet more technical due to the many boundary terms; their proofs are deferred to the Appendix. Recall Lemma~\ref{lem_AF_id} regarding the short-time second moment estimates.

\begin{lem}\label{Moment2}
For any $V \in L^1(E; \mathfrak{m})\cap \mathcal S_D^1(X)$, we have for any  $x \in E$:
\begin{equation*}
	\lim_{t\to \infty}{\mathbb{E}}_{x|t}^{\mu, F}\left[\left(\frac{1}{t}A_t^V\right)^2\right] =\left(\int_E V\phi_0^2\d\mathfrak{m}\right)^2. 
\end{equation*}
\end{lem}
\vskip 0.2cm
We need an additional condition $N[G^2], N[(G^2)^T]\in L^1(E;\mathfrak{m})\cap \mathcal S^1_D(X)$ to prove Lemma \ref{JumpWLLN} below because we do not assume the boundedness of $G$.

\begin{lem}\label{JumpWLLN}
For any $G\in \mathcal K(E\times E)$ such that $N[|G|], N[|G|^T], N[G^2], N[(G^2)^T]\in L^1(E; \mathfrak{m}) \cap \mathcal S_D^1(X)$,
we have for any $x \in E$:
\begin{equation*}
	\lim_{t\to \infty}{\mathbb{E}}_{x|t}^{\mu, F}\left[\left(\frac{1}{t}A_t^G\right)^2\right] 
	=\left(\iint_{E\times E}G(y,z){\mathcal J}_{\phi_0}({\rm d}y{\rm d}z)\right)^2.
\end{equation*} 
\end{lem}
\vskip 0.2cm

Now, we deduce the following conditional functional weak law of large numbers for $X$ under ${\mathbb{P}}_{x|t}^{\mu, F}$. 

\begin{thm}\label{VGWLLN}
For any $V$ and $G\in \mathcal K(E\times E)$ such that $V, N[|G|],N[|G|^T], N[G^2], N[(G^2)^T] \in L^1(E;\mathfrak{m})\cap {\mathcal S}_D^1(X)$,
we have for any $x \in E$:
\begin{equation*}
	\lim_{t \to \infty}{\mathbb{P}}_{x|t}^{\mu, F}\left(\left|
	\frac{1}{t}A_t^{V,G} - \left(\int_E V\phi_0^2 \d\mathfrak{m} 
	+\iint_{E\times E}G(y,z){\mathcal J}_{\phi_0}({\rm d}y{\rm d}z)\right)
	\right| \ge \varepsilon\right)=0.
\end{equation*}
\end{thm}

\begin{proof}
By Lemma \ref{Moment2}, Lemma \ref{JumpWLLN} and Theorem~\ref{JumpErgodic}, we see that
\begin{equation*}
	\begin{split}
		&{\mathbb{E}}_{x|t}^{\mu, F}\left[\left|\frac{1}{t}A_t^{V,G} - \left(\int_E V\phi_0^2\d\mathfrak{m} +\iint_{E\times E}G(y,z){\mathcal J}_{\phi_0}({\d}y{\d}z)\right)\right|^2\right] \\
		&\qquad \le 2{\mathbb{E}}_{x|t}^{\mu, F}\left[\left|\frac{1}{t}A_t^{V} -\int_E V\phi_0^2\d\mathfrak{m}\right|^2\right] + 2{\mathbb{E}}_{x|t}^{\mu, F}\left[\left|\frac{1}{t}A_t^G -\iint_{E\times E} G(y,z){\mathcal J}_{\phi_0}({\d}y{\d}z)\right|^2\right] \\
		&\qquad \longrightarrow~~ 0 \qquad \text{as}~~ t \to \infty.
	\end{split}
\end{equation*}
Hence, the assertion immediately follows from Chebyshev's inequality.
\end{proof}

\section{Large deviation for additive functionals}\label{sec:LDP}  

In this section, we shall consider the measure $\mu$ as $\mu({\d}x)=\theta V(x)\mathfrak{m}({\d}x)$ for a Borel measurable function $V$ on $E$ and $\theta\in \mathbb R^1$. Our previous assumption that $\mu \in \mathcal S_K^1(X)$ is then equivalent to $V\in \mathcal S_K^1(X)$ and we recall that $F\in \mathcal K(E\times E)$ is assumed to be bounded symmetric and such that $N[|F|] \in \mathcal S_K^1(X)$. We exclude the trivial case for $F$, i.e. we suppose $F\not\equiv 0$.
In the sequel, we assume in addition that $V \in L^1(E; \mathfrak{m})$ and $N[|F|] \in L^1(E;\mathfrak{m})$. It directly entails that $N[|F|^T], N[F^2], N[(F^2)^T] \in \mathcal S_K^1(X) \cap L^1(E;\mathfrak{m})$ by the boundedness of $F$.	
The aim of this section is to study the large deviation principle for the pairs of continuous and purely discontinuous additive functionals $A_t^{V,F}$. The weak conditional law of large numbers for the special Feynman-Kac functional $W^{V,F}_t=\exp (-A_t^{V,F}){\bf 1}_{\{t<\zeta\}}$ helps to establish the large deviation principle for $A_t^{V,F}$ by using a rate function in a more direct representation via spectral functions. 
\vskip 0.2cm
For $\theta\in \mathbb{R}^1$, let 
\begin{equation}\label{eq_lbdMin}
\lambda_0(\theta):=\lambda_0({\theta V,\theta F})=\inf \left\{{\mathcal E}^{\theta V,\theta F}(u,u) : u \in {\mathcal D}({\mathcal E}), \int_{E}u^2{\d}\mathfrak{m}=1\right\},
\end{equation}
where
\begin{align*}
\mathcal E^{\theta V, \theta F}(u,u)&:={\mathcal E}(u,u) + \theta \int_Eu(x)^2V(x)\mathfrak{m}({\d}x) + \iint_{E\times E}u(x)u(y)\left(1-e^{-\theta F(x,y)}\right)N(x,{\d}y)\mathfrak{m}({\d}x).
\end{align*}
By virtue of Theorem \ref{GS}, there exists a ground state $\phi_0^{(\theta)}:=\phi_0^{\theta V, \theta F}$ of the bilinear form $(\mathcal E^{\theta V, \theta F}, \mathcal D(\mathcal E))$, that is, 
\begin{equation}
\int_E \phi_0^{(\theta)}(x)^2\mathfrak{m}({\rm d}x)=1 \quad \text{and}\quad \lambda_0(\theta)={\mathcal E}^{\theta V, \theta F}\left(\phi_0^{(\theta)},\phi_0^{(\theta)}\right).
\label{eq_def_lE}
\end{equation} 

We make the following assumption that is derived from {\bf (A3)},  where we focus on $\theta \in \mathbb{R}^1_- = (-\infty, 0)$ in order to study deviations where $A^{V, G}_t$ takes larger values than expected:
\begin{enumerate}
\item[{\bf (A3)$_{\theta}$}] For any $\theta \in \mathbb{R}^1_-$, the Feynman--Kac semigroup $\{p_t^{\theta V, \theta F}\}_{t\ge 0}$ is (IUC), that is, there exist constants $c_t(\theta)>0$ such that for all $t>0$ and $x,y \in E$,
\begin{equation*}
p_t^{\theta V,\theta F}(x,y) \le c_t(\theta) \phi_0^{(\theta)}(x)\phi_0^{(\theta)}(y).
\end{equation*}
\end{enumerate}
\noindent
Note that under {\bf (A3)$_{\theta}$}, 
$p_t^{\theta V, \theta F}$ is a Hilbert-Schmidt operator on $L^2(E;\mathfrak{m})$.
Therefore, $p_t^{\theta V, \theta F}$ is a compact operator and thus has a discrete spectrum for any $\theta \in \mathbb{R}^1_-$. 
Set 
\begin{equation*}
C_{V,F}(\theta):=-\lambda_0(\theta).
\end{equation*} 
The bilinear form  associated with $\{p_t^{\theta V, \theta F}\}_{t\ge 0}$ is $(\mathcal E^{\theta V, \theta F}, {\mathcal D}({\mathcal E}))$ which is an analytic function in $\theta$. These bilinear forms constitute a holomorphic family of type (A) in the terminology of \cite[p.395]{Kato:1982}. Thus, by \cite[Chapter VII; Theorems 1.8 and 4.2]{Kato:1982}, the principal $L^2$-eigenvalue $\lambda_0(\theta)$ of $(\mathcal E^{\theta V, \theta F}, {\mathcal D}({\mathcal E}))$ is differentiable in $\theta$ by the analytic perturbation theory and so is $C_{V,F}(\theta)$. 
\vskip 0.2cm
For $\vartheta \in \mathbb{R}^1_-$ and $u \in \mathcal D(\mathcal E)$, set ${\tt t}(\vartheta)[u]:=-{\mathcal E}^{\vartheta V, \vartheta F}(u,u)$. By the Taylor expansion at $\vartheta=\theta$, we see that ${\tt t}(\vartheta)[u]$ can be expressed as 
\begin{align*}
{\tt t}(\vartheta)[u]&=-{\mathcal E}(u,u)-\vartheta \int_E u^2V{\d}\mathfrak{m} - \iint_{E\times E}u(x)u(y)\left(1-e^{-\vartheta F(x,y)}\right)N(x,{\d}y)\mathfrak{m}({\d}x) \\
&=-{\mathcal E}(u,u)-\theta \int_E u^2V{\d}\mathfrak{m} - \iint_{E\times E}u(x)u(y)\left(1-e^{-\theta F(x,y)}\right)N(x,{\d}y)\mathfrak{m}({\d}x) \\
&\qquad \quad - (\vartheta-\theta)\left(\int_Eu^2V{\d}\mathfrak{m}+\iint_{E\times E}F(x,y)u(x)u(y)e^{-\theta F(x,y)}N(x,{\d}y)\mathfrak{m}({\d}x)\right) \\
&\qquad \quad + \sum_{n=2}^\infty (\vartheta-\theta)^n \iint_{E\times E}\frac{(-1)^nF^n(x,y)}{n!}u(x)u(y)e^{-\theta F(x,y)}N(x,{\d}y)\mathfrak{m}({\d}x)\\
&:={\tt t}^{(0)}(\theta)[u]+(\vartheta-\theta){\tt t}^{(1)}(\theta)[u] + \sum_{n=2}^\infty (\vartheta-\theta)^n {\tt t}^{(n)}(\theta)[u],
\end{align*}
By virtue of \cite[Chapter VII (4.44)]{Kato:1982}, the first differential coefficient of $C_{V,F}(\vartheta)$ at $\vartheta=\theta$ is given as follows:
\begin{equation}\label{Derivative}
\begin{split}
C_{V,F}'(\theta) &= {\tt t}^{(1)}(\theta)[\phi_0^{(\theta)}] \\
&=-\int_E V(x)\phi_0^{(\theta)}(x)^2 \mathfrak{m}({\rm d}x) - \iint_{E\times E}F(x,y)\phi_0^{(\theta)}(x)\phi_0^{(\theta)}(y)e^{-\theta F(x,y)}N(x,{\rm d}y)\mathfrak{m}({\rm d}x).
\end{split}
\end{equation}
Note that $C_{V,F}$ is a function that is not only convex
(cf \cite[Lemma 5.5]{CT}), but actually strictly convex as stated in the next lemma.
Thus, $C_{V,F}'(\theta)$ is a strictly increasing function in $\mathbb{R}^1_-$.

\begin{lem}
\label{str_conv}
Suppose that {\bf (A1)}, {\bf (A2)} and {\bf (A3)$_{\theta}$} hold.
Then, the function $\theta\mapsto C_{V,F}(\theta)$ is strictly convex on $\mathbb{R}^1_-$.
\end{lem} 
\begin{proof}
For any $\theta_1, \theta_2\in \mathbb{R}^1_-$ and any fixed $\xi \in (0, 1)$, let $\theta_* = \xi \theta_1 + (1-\xi) \theta_2$.
Since $\phi_0^{(\theta)}$ is the minimizer of \eqref{eq_lbdMin}, $\mathcal E^{\theta V, \theta F}(\phi_0^{(\theta)},\phi_0^{(\theta)}) \le \mathcal E^{\theta V, \theta F}(\phi_0^{(\theta_*)},\phi_0^{(\theta_*)})$ for any $\theta \in \mathbb{R}^1_-$. So we have
\begin{equation*}
\xi\, C_{V,F}(\theta_1) +(1-\xi)\, C_{V,F}(\theta_2)
\ge \xi\, {\tt t}(\theta_1)\left[\phi_0^{(\theta_*)}\right] 
+ (1-\xi)\, {\tt t}(\theta_2)\left[\phi_0^{(\theta_*)}\right].
\end{equation*}
The right-hand side above has the expression below: 
\begin{align*}
	&-{\mathcal E}\left(\phi_0^{(\theta_*)},\phi_0^{(\theta_*)}\right)-\theta_* \int_E V(x)\phi_0^{(\theta_*)}(x)^2\mathfrak{m}({\rm d}x) 	\\
	&\qquad \qquad \qquad \quad ~~~- \iint_{E\times E}\phi_0^{(\theta_*)}(x)\phi_0^{(\theta_*)}(y)\left(1-\xi\,e^{-\theta_1 F(x,y)}- (1-\xi)\,e^{-\theta_2 F(x,y)}\right)N(x,{\d}y)\mathfrak{m}({\d}x).
\end{align*}
Because $\theta \mapsto e^{-\theta F}$ is strictly convex in $\theta$ and $\phi_0^{(\theta_*)}>0$, 
\begin{align*}
	&\xi\, C_{V,F}(\theta_1) +(1-\xi)\, C_{V,F}(\theta_2) \\
	&\quad > -{\mathcal E}\left(\phi_0^{(\theta_*)},\phi_0^{(\theta_*)}\right)-\theta_* \int_E V(x)\phi_0^{(\theta_*)}(x)^2\mathfrak{m}({\rm d}x) \\
	&\quad \qquad \qquad \qquad \quad \quad ~~- \iint_{E\times E}\phi_0^{(\theta_*)}(x)\phi_0^{(\theta_*)}(y)\left(1-e^{-\theta_* F(x,y)}\right)N(x,{\d}y)\mathfrak{m}({\d}x) \\
	&\quad = -\mathcal E^{\theta_* V, \theta_* F}\left(\phi_0^{(\theta_*)},\phi_0^{(\theta_*)}\right)  = C_{V,F}(\xi \theta_1 + (1-\xi)\theta_2).
\end{align*}
It concludes the proof of Lemma~\ref{str_conv}.
\end{proof}

Let $\Psi_{V,F}(\theta)$ be the function given by
\begin{equation*}
\Psi_{V,F}(\theta)
=\int_E V(x)\phi_0^{(\theta)}(x)^2\mathfrak{m}({\rm d}x) + \iint_{E\times E}F(x,y){\mathcal J}_{\phi_0^{(\theta)}}({\d}x{\d}y).
\end{equation*}
By \eqref{Derivative}, we see that $\Psi_{V,F}(\theta)=-C_{V,F}'(\theta)$. Hence $\Psi_{V,F}(\theta)$ is strictly decreasing and continuous on $\R^1$. Denote by $\Psi_{V,F}^{-1}$ the inverse function of $\Psi_{V,F}$.  
Set $\Psi_{V,F}(\mathbb{R}_-^1):=\{\Psi_{V,F}(\theta) : \theta \in \mathbb{R}_-^1\}$ and write $\Psi_{V,F}(\mathbb{R}_-^1)^o$ as the interior of $\Psi_{V,F}(\mathbb{R}_-^1)$. We  note that $\Psi_{V,F}(\mathbb{R}_-^1)^o\neq \emptyset$ because of the strict convexity of $C_{V,F}$.
\vskip 0.2cm
Now, we have the following large deviation for $A_t^{V,F}$.

\begin{thm}\label{LDPforVF}
Suppose that {\bf (A1)}, {\bf (A2)} and {\bf (A3)$_{\theta}$} hold. 
Then, for any $\gamma \in \Psi_{V,F}(\mathbb{R}_-^1)^o$ and any $x \in E$,
\begin{equation}\label{LDexpress}
\lim_{t\to \infty}\frac{1}{t}\log {\mathbb{P}}_x\left(\frac{A_t^{V,F}}{t} \in [\gamma, \infty),~t< \zeta\right)
=C_{V,F}(\theta_\gamma)-\theta_\gamma C_{V,F}'(\theta_\gamma),
\end{equation}
where $\theta_\gamma$ is the non-positive real number given by $\theta_\gamma=\Psi^{-1}_{V,F}(\gamma)$.  
\end{thm}

\begin{proof}
First, we prove the upper bound. For $\theta \in \mathbb{R}^1_-$, let $W_t^{(\theta)} := W_t^{\theta V, \theta F} = \exp (-\theta A^{V,F}_t)\mathbf{1}_{\{ t<\zeta \}}$. By \eqref{vFt_cvg}, the logarithmic moment generating function of $-\theta A_t^{V,F}$ has the limit $C_{V,F}(\theta)$:
$	\lim_{t\to \infty}\frac{1}{t}\log {\mathbb{E}}_x\left[W_t^{(\theta)}\right]=C_{V,F}(\theta).$

Since the function $\Psi_{V,F}(\theta)$ is strictly decreasing and continuous on $\mathbb{R}^1_-$, there exists $\theta_\gamma \in \mathbb{R}_-^1$ such that $\gamma=\Psi_{V,F}(\theta_\gamma)$ for any $\gamma \in \Psi_{V,F}(\mathbb{R}_-^1)^o$. Then, by the G\"artner-Ellis theorem (\cite{DZ:1998}), we see that for any $x \in E$, 
\begin{align*}
\begin{split}
	\limsup_{t\to \infty}\frac{1}{t}\log {\mathbb{P}}_x\left(\frac{A_t^{V,F}}{t} \in [\gamma, \infty),~t<\zeta\right) &\,=\limsup_{t\to \infty}\frac{1}{t}\log {\mathbb{P}}_x\left(\frac{-A_t^{V,F}}{t} \in (-\infty,-\gamma],~t<\zeta\right) \\
	&\le -\inf_{\lambda \in (-\infty,-\gamma]} \sup_{\theta \in \mathbb{R}^1_-}\left\{\lambda \theta - C_{V,F}(\theta)\right\} \\
	&\le \sup_{\lambda \in (-\infty, -\Psi_{V,F}(\theta_\gamma)]}\left\{C_{V,F}(\theta_\gamma)-\lambda \theta_\gamma\right\} \\
	&\le C_{V,F}(\theta_\gamma) + \theta_\gamma\Psi_{V,F}(\theta_\gamma) \\
	&= C_{V,F}(\theta_\gamma) - \theta_\gamma C_{V,F}'(\theta_\gamma).
\end{split}
\end{align*}
Next, we turn to the proof of the lower bound. Let $\{\theta_n\} \subset \mathbb{R}_-^1$ be a sequence such that $\theta_n \uparrow \theta_\gamma$ as $n\to \infty$. Put $\varepsilon_n:=\Psi_{V,F}(\theta_n)-\Psi_{V,F}(\theta_\gamma)=\Psi_{V,F}(\theta_n) -\gamma >0$. Then, we have for large enough $t>0$,
\begin{equation*}
\begin{split}
	&{\mathbb{P}}_x\left(\frac{A_t^{V,F}}{t} \in [\gamma,\infty),~t < \zeta\right) \ge {\mathbb{P}}_x\left(\frac{A_t^{V,F}}{t} \in [\gamma, \gamma+2\varepsilon_n],~t<\zeta\right) \\
	& \qquad=e^{C_{V,F}(\theta_n)\,t}e^{-C_{V,F}(\theta_n)\,t} {\mathbb{E}}_x\left[e^{\theta_n A_t^{V,F}}e^{-\theta_n A_t^{V,F}}~;~\frac{A_t^{V,F}}{t} \in [\gamma, \gamma+2\varepsilon_n],~t<\zeta \right]\\
	&	\qquad \ge e^{\big(C_{V,F}(\theta_n)+(\gamma+2\varepsilon_n) \theta_n\big) t}\; e^{\lambda_0 (\theta_n)t} {\mathbb{E}}_x\left[e^{-\theta_n A_t^{V,F}} ;~\frac{A_t^{V,F}}{t} \in \big[\Psi_{V,F}(\theta_n)-\varepsilon_n, \Psi_{V,F}(\theta_n)+\varepsilon_n\big],~t<\zeta\right]. 
\end{split}
\end{equation*}
By applying Theorem \ref{VGWLLN} in combination with Lemma~\ref{lem_AF_id} with $\mu=\theta_n V$, $F=\theta_n F$ and $G=F$, and using the fact that
\begin{equation*}
e^{\lambda_0(\theta_n)t} {\mathbb{E}}_x\left[e^{-\theta_n A_t^{V,F}} ;~t<\zeta\right]=\widehat{p}_t^{\,\theta_n V,\,\theta_n F}{\bf 1}(x)
~~ \stackrel{t\to \infty}{\longrightarrow}~~ 
\phi_0^{(\theta_n)}(x)\int_E \phi_0^{(\theta_n)}{\rm d}\mathfrak{m},
\end{equation*} 
one can see
\begin{equation*}
e^{\lambda_0(\theta_n)t}
	{\mathbb{E}}_x\left[e^{-\theta_n A_t^{V,F}} ;~\frac{A_t^{V,F}}{t} \in \big[\Psi_{V,F}(\theta_n)-\varepsilon_n, \Psi_{V,F}(\theta_n)+\varepsilon_n\big],~t<\zeta\right]
~~ \stackrel{t\to \infty}{\longrightarrow}~~ 
\phi_0^{(\theta_n)}(x)\int_E \phi_0^{(\theta_n)}{\rm d}\mathfrak{m}.
\end{equation*}
Hence we have
\begin{equation*}
\liminf_{t\to \infty}\frac{1}{t}\log {\mathbb{P}}_x\left(\frac{1}{t}A_t^{V,F} \in [\gamma,\infty),~ t<\zeta\right) \ge C_{V,F}(\theta_n) + (\gamma+2\varepsilon_n)\,\theta_n.
\end{equation*}
Letting $n\to \infty$, the right-hand side of the above converges to $C_{V,F}(\theta_\gamma) + \gamma \,\theta_\gamma$ because $\varepsilon_n =\Psi_{V,F}(\theta_n) -\gamma=\Psi_{V,F}(\theta_n) - \Psi_{V,F}(\theta_\gamma) \to 0$ and $C_{V,F}(\theta_n) \to C_{V,F}(\theta_\gamma)$ as $n\to \infty$ by the continuities of $\Psi_{V,F}(\theta)$ and $C_{V,F}(\theta)$, respectively. This leads us to the lower bound  given by $C_{V,F}(\theta_\gamma) -\theta_\gamma C_{V,F}'(\theta_\gamma)$. 
\end{proof}

We can represent the rate $C_{V,F}(\theta_\gamma) -\theta_\gamma C_{V,F}'(\theta_\gamma)$ in \eqref{LDexpress} in a more direct way via a bilinear form:

\begin{cor}\label{LDPrem}
Suppose that {\bf (A1)}, {\bf (A2)} and {\bf (A3)$_{\theta}$} hold.
Then, for any $\gamma \in \Psi_{V,F}(\mathbb{R}_-^1)^o$ and any $x \in E$,
\begin{equation*}
\lim_{t\to \infty}\frac{1}{t}\log {\mathbb{P}}_x\left(\frac{A_t^{V,F}}{t} \in [\gamma, \infty),~t<\zeta\right)
=-{\mathcal A}\left(\phi_0^{(\theta_\gamma)},\phi_0^{(\theta_\gamma)}\right),
\end{equation*}
where $\theta_\gamma$ is the non-positive real number given by $\theta_\gamma=\Psi^{-1}_{V,F}(\gamma)$ and 
\begin{align*}
{\mathcal A}\left(\phi_0^{(\theta_\gamma)},\phi_0^{(\theta_\gamma)}\right)&={\mathcal E}\left(\phi_0^{(\theta_\gamma)},\phi_0^{(\theta_\gamma)}\right) \\
&\quad + \iint_{E\times E}\phi_0^{(\theta_\gamma)}(x)\phi_0^{(\theta_\gamma)}(y)\left(1-e^{-\theta_\gamma F(x,y)}-\theta_\gamma F(x,y) e^{-\theta_\gamma F(x,y)}\right)N(x,{\d}y)\mathfrak{m}({\d}x).
\end{align*}  
\end{cor}

\section{Examples}\label{sec:Example}  

We give some examples satisfying the assumption {\bf (A3)}, by relying on the framework proposed in \cite{Kaleta:2015}.
 Examples satisfying {\bf (A3)$_\theta$} are directly derived from them.
\vskip 0.1cm
Let $X=(X_t, {\mathbb{P}}_x)$ be a symmetric L\'evy process in $\mathbb{R}^d$ with $d\ge 1$, that is, ${\bf X}$ is a symmetric Markov process satisfying the strong Markov property and has right continuous paths with left limits (c\`adl\`ag paths). It is known that $X$ is completely determined by its characteristic exponent $\psi$ given by the L\'evy-Khintchine formula: for $\xi \in \mathbb{R}^d$, ${\mathbb{E}}_0[e^{i\xi \cdot X_t}]=e^{-t\psi(\xi)}$ holds with
\begin{equation*}
\psi(\xi)=S\xi \cdot \xi + \int_{\mathbb{R}^d}(1-\cos (\xi\cdot z))J({\d}z) \ge 0,
\end{equation*}
where $S$ is a symmetric nonnegative definite $d\times d$ matrix and $J$ is a Radon measure on $\mathbb{R}^d \setminus \{0\}$ such that $\int_{\mathbb{R}^d}(1\wedge |z|^2)J({\d}z) <\infty$, which is called the L\'evy measure. If there exists a density $J(z)$ such that $J({\d}z)=J(z){\d}z$, then we call it the L\'evy (jump) intensity of $X$. In particular, when $S=0$ and $J\neq 0$, the process $X$ is said to be a pure jump L\'evy process. It is known that if 
$$
\int_{\mathbb{R}^d}e^{-t\psi(\xi)}{\d}\xi <\infty, \quad \text{for all~ $t>0$},\leqno{\text{\bf (B1)}}
$$
then $X$ possesses {\bf (I)} and {\bf (SF)} (see \cite[Proposition 5]{KS:2013}). So there is a probability density $p_t(x,y)=p_t(y-x,0):=p_t(y-x)$ of $X$. In addition, {\bf (B1)} implies (UC) of the transition semigroup $\{p_t\}_{t\ge 0}$ of $X$. Further, there also exists a probability density  $p_t^D(x,y)$ of $X$ killing upon exiting an open bounded set $D \subset \mathbb{R}^d$ (see \cite[(2.1)]{Kaleta:2015}). 
\vskip 0.1cm
Let us make the following conditions. Here we will use the notation $f \asymp Cg$ which means that $C^{-1}g \le f \le Cg$ with a constant $C$.
\vskip 0.2cm
\noindent
{\bf (B2)}~~There exists a strictly positive L\'evy intensity $J(x)$ of $X$ such that
\begin{enumerate}
\item[(a)] For every $0< r \le 1/2$, there is a constant $C_1=C_1(r) \ge 1$ such that $J(x) \asymp C_1J(y)$ for $r \le |y| \le |x| \le |y|+1$.
\item[(b)] There is a constant $C_2 \ge 1$ such that $J(x) \le C_2J(y)$ for $1/2 \le |y| \le |x|$

\item[(c)] There is a constant $C_3 \ge 1$ such that
\begin{equation*}
\int_{|z-x|>1/2, |z-y|>1/2}J(x-z)J(z-y){\rm d}z \le C_3J(x-y) \quad \text{for}~~|y-x| \ge 1.
\end{equation*}
\end{enumerate}

Let $R^D(x,y)$ be the Green function of $X$ on $D$, $R^D(x,y):=\int_0^\infty p_t^D(x,y){\d}t$ for all $x,y \in D$ and $R^D(x,y)=0$ for $x \notin D$ or $y \notin D$. We make further condition on the Green function of $X$: 
\vskip 0.2cm
\noindent
{\bf (B3)}~~For all $0 < p < q <R\le 1$,  $\sup_{x \in B(0,p)}\sup_{y \in B(0,q)^c}R^{B(0,R)}(x,y) < \infty$.
\vskip 0.2cm
\noindent
Note that the conditions {\bf (B1)}, {\bf (B2)} and {\bf (B3)} are satisfied by a wide class of symmetric L\'evy processes (see \cite[$\S$4]{Kaleta:2015} for some specific classes of symmetric L\'evy processes satisfying these conditions). In the following, let us assume that $X$ satisfies  {\bf (B1)}, {\bf (B2)} and {\bf (B3)}.
\vskip 0.1cm
For a Borel function $V$ on $\mathbb{R}^d$ and a symmetric bounded Borel function $F(x,y)$ on $\mathbb{R}^d \times \mathbb{R}^d$ that vanishes along the diagonal, let $V_F(x):=V(x)+N[1-e^{-F}](x)$. Assume that $V$ and $F$ satisfy the following conditions:
\begin{enumerate}
\item[(d1)] $V \in \mathcal S_K^1(X)$ and $N[|F|] \in \mathcal S_K^1(X)$.
\item[(d2)] $V_F(x) \to \infty$ as $|x| \to \infty$. 
\end{enumerate}
Let us consider the multiplicative functional
\begin{equation*}
Y_t=\exp \left(-A^F_t+\int_0^t N\left[1-e^{-F}\right](X_s){\d}s\right), \quad t<\zeta,
\end{equation*}
and define $\widetilde{X}=(X_t, \widetilde{\mathbb{P}}_x)$ thanks to it 
as the transformed process of $X$ 
by the compensated pure jump Girsanov transform (as in \cite[Section~6]{FOT}, or \cite[Section~62]{Sha:Book}). 
The transition semigroup $\widetilde{p}_t$ of $\widetilde{X}$ is given by $\widetilde{p}_tf(x)=\widetilde{\mathbb{E}}_x[f(X_t), t<\zeta]={\mathbb{E}}_x[Y_tf(X_t), t<\zeta]$ for $f \in \mathfrak{B}_b(\mathbb{R}^d)$. 
The boundedness of $F$ implies that the L\'evy intensity $\widetilde{J}(x)$ of $\widetilde{X}$ satisfies $\widetilde{J}(x) \asymp CJ(x)$. Let us assume that the Green function $\widetilde{R}^{B(0,R)}(x,y)$ of $\widetilde{X}$ satisfies
\begin{enumerate}
\item[(d3)] $\widetilde{R}^{B(0,R)}(x,y) \asymp CR^{B(0,R)}(x,y)$~ for any $x, y\in E$.
\end{enumerate}
For instance, this holds by the conditional gaugeability for $(B(0,R),F)$ (cf. \cite{CS, Song:2006}) when $X$ is a symmetric $\alpha$-stable process with $0<\alpha<2$ and $F$ satisfying 
\begin{equation}
|F(x,y)| \le c|x-y|^\beta,~ (x,y) \in D\times D~~\text{and}~~F(x,y)=0, ~\text{otherwise},
\end{equation}
where $D$ is a compact set of $\mathbb{R}^d$, $c$ and $\beta$ are two positive constants such that $\beta >\alpha$.
Then, under the assumption (d1) above, we see that $\widetilde{X}$ also satisfies the conditions {\bf (B1)}, {\bf (B2)} and {\bf (B3)}. We note by virtue of \cite[Lemma 4.5]{KimKuwae:TAMS} that $V \in \mathcal S_K^1(\widetilde{X})$ and $N[|F|] \in \mathcal S_K^1(\widetilde{X})$ under the assumption (d1). In particular, $V_F \in \mathcal S_K^1(\widetilde{X})$ by the boundedness of $F$.
Further, since the Feynman-Kac semigroup $p_t^{V,F}$ can be expressed as  
\begin{equation*}
p_t^{V,F}f(x)
=\widetilde{\mathbb{E}}_x\left[e^{-\int_0^t V_F(X_s){\d}s}
 f(X_t),~t<\zeta\right]
 :=\widetilde{p}_t^{V_F}f(x), \quad f \in \mathfrak{B}_b(\mathbb{R}^d),
\end{equation*}
we see by the assumption (d1) with an argument based on the proof of \cite[Corollary 3.2]{KKT:2016} (or \cite{KimKuwae:TAMS}) that $p_t^{V,F}$ is a bounded operator on $L^2(\mathbb{R}^d)$ for all $t>0$ and also satisfies {\bf (I)} and {\bf (SF)}. In addition, the operator $p_t^{V,F}$ is compact on $L^2(\mathbb{R}^d)$ for any $t>0$ under the assumption (d2). Thus, 
it follows from the general theory of semigroups that there exists a normalized principal eigenfunction $\phi_0(x):=\phi_0^{V,F}(x)$ on $L^2(\mathbb{R}^d)$, which is bounded continuous and such that $p_t^{V,F}\phi_0 = e^{-\lambda_0 t}\phi_0$ for any $t>0$, where $\lambda_0:=\lambda_0(V,F)$ is a principal eigenvalue of finite multiplicity. $\phi_0$ is called the ground state and can be assumed to be strictly positive in view of {\bf (I)}. Moreover, by virtue of \cite[Theorem 2.7]{Kaleta:2015}, with the the following assumption 
\begin{enumerate}
\item[(d4)] $\lim_{|x|\to \infty}\frac{V_F(x)}{|\log J(x)|}=\infty$,
\end{enumerate} 
the Feynman-Kac semigroup $p_t^{V,F}$ is ground state dominated, that is, there is a constant $C_4>0$ such that for all $t>0$ 
\begin{equation}\label{eq_def_GSD}
p_t^{V,F}1(x) \le C_4 \phi_0(x), \quad x \in \mathbb{R}^d.
\end{equation}
This property of ground state domination \eqref{eq_def_GSD} is equivalent,  under {\bf (B1)}, to the intrinsic ultracontractivity 
of $\{p_t^{V,F}\}_{t\ge 0}$, which we denoted as assumption {\bf (A3)}.

Clearly, if the assumptions (d1), (d2) and (d4) hold for $V$ and $F$, and the assumption (d3) holds for $F$, they also hold for $\theta V$ and $\theta F$ for any $\theta > 0$.
Hence, $\{p_t^{\theta V, \theta F}\}_{t>0}$ is (IUC) for any $\theta > 0$.
In the case of $\theta < 0$,  as required in {\bf (A3)$_\theta$},
(IUC) of $\{p_t^{\theta V, \theta F}\}_{t>0}$ holds when we suppose (d2) and (d4) with $-\infty$ in place of $\infty$.

\appendix
\section{Appendix}\label{appn} 

In this appendix, we give the proofs of Lemma \ref{Moment2} and Lemma \ref{JumpWLLN}.
\vskip 0.1cm
\begin{proof}[Proof of Lemma~\ref{Moment2}]
		We first assume that $V$ is nonnegative. The equalities below follow from (3.3) 
		for any $x \in E$, with the change of variables $s= pt$ and $s+r =qt$:
		\begin{equation}\label{eq_dec_V2}
			\begin{split}
				{\mathbb{E}}_{x|t}^{\mu, F}\left[\left(\frac{1}{t}A_t^V\right)^2\right] &=\frac{2\varPhi_t(x)^{-1}}{t^2}\int_0^t \int_0^{t-s} \widehat p^{\mu, F}_s\left(V\cdot \widehat p^{\mu, F}_r(V\cdot \varPhi_{t-s-r})\right)(x) {\d}r{\d}s \\
				&=2\int_0^1\int_p^1 \varPhi_t(x)^{-1}
				\widehat p^{\mu, F}_{pt}\left(V\cdot \widehat p^{\mu, F}_{(q-p)t}\left(V\cdot \varPhi_{(1-q) t}\right)\right)(x){\d}q{\d}p.
			\end{split}
		\end{equation}
		For any $\epsilon>0$, we deduce the following upper-bound of the integrant for large enough $t>0$ such that $t \ge \max\{\varepsilon/p, \varepsilon/(q-p), \varepsilon/(1-q)\}$, similarly as for \eqref{eq_L1_bound}:
		\begin{equation*}
			\left\|\varPhi_t^{-1}\,\widehat p^{\mu, F}_{pt}\left(V\cdot \widehat p^{\mu, F}_{(q-p)t}\left(V\cdot \varPhi_{(1-q) t}\right)\right)\right\|_{L^\infty (E;\phi_0^2\mathfrak{m})} \le 2 c_\varepsilon^3  \int \phi_0\d\mathfrak{m} \cdot \|V\|^2_{L^1(E; \phi_0^2\mathfrak{m})}\,.
		\end{equation*}
		By applying Lemma~\ref{QED1}$(2)$ 
		with $f=g=V$, and the dominated convergence theorem, we see then
		\begin{equation}\label{eq_lim_wo_bnd}
			\left(\int_E V\phi_0^2\d\mathfrak{m}\right)^2
			= 2 \lim_{t\to \infty}\int_{\varepsilon/t}^{1-\varepsilon/t}\int_{
				p + \varepsilon/t}^{1-\varepsilon/t}\varPhi_t(x)^{-1}\widehat p^{\mu, F}_{pt}\left(V\cdot \widehat p^{\mu, F}_{(q-p)t}\left(V\cdot \varPhi_{(1-q) t}\right)\right)(x) {\d}q\, {\d}p\,.
		\end{equation}
For a signed function $V$ that is decomposed between its positive part  $V_+ = (|V| + V)/2$ and its negative part $V_- = (|V| - V)/2$,
			we see that the same reasoning justifies \eqref{eq_lim_wo_bnd} for $V_+, V_-$ and $V_++V_-$ instead of $V$.
			Due to the identity $(a-b)^2 = 2a^2 + 2b^2 - (a+b)^2$,
			applied to $a= \int_E V_+\phi_0^2\d\mathfrak{m}$
			and $b= \int_E V_-\phi_0^2\d\mathfrak{m}$,
			and  linearity properties,
			we deduce \eqref{eq_lim_wo_bnd} for any signed function $V \in L^1(E;\phi_0^2\mathfrak{m})$.

	We shall prove that the error between the above integral and the similar one in \eqref{eq_dec_V2}
			gets negligible as $t$ tends to infinity. 
			We first prove this property in what follows for functions $V$ that are nonnegative. We will then apply this result to $|V|$ instead of $V$ to extend the convergence for any signed function $V$.

		\vskip 0.1cm
			There are six boundary terms that are associated with the following subsets of $[0, 1]^2$: 
			\begin{equation*}
				\begin{split}
					&\mathfrak B^1(\varepsilon/t):=\left\{(p, q) \mid p\in [0, \varepsilon/t],~ q-p\ge \varepsilon/t,~ \text{and}~ 1-q\ge \varepsilon /t\right\}, \\
					&\mathfrak B^2(\varepsilon/t):=\left\{(p, q) \mid p\ge \varepsilon/t,~ q-p\ge \varepsilon/t,~ \text{and}~ 1-q\in [0, \varepsilon/t]\right\}, \\
					&\mathfrak B^3(\varepsilon/t):=\left\{(p, q) \mid p\in [0, \varepsilon/t],~ q-p\ge \varepsilon/t,~ \text{and}~ 1-q\in [0, \varepsilon/t]\right\}, \\
					&\mathfrak B^4(\varepsilon/t):=\left\{(p, q) \mid p\in [0, \varepsilon/t],~ q-p\in [0, \varepsilon/t],~ \text{and}~ 1-q\ge \varepsilon/t\right\}, \\
					&\mathfrak B^5(\varepsilon/t):=\left\{(p, q) \mid p\ge \varepsilon/t,~ q-p\in [0, \varepsilon/t],~ \text{and}~ 1-q\ge \varepsilon/t\right\}, \\
					&\mathfrak B^6(\varepsilon/t):=\left\{(p, q) \mid p\ge \varepsilon/t,~ q-p\in [0, \varepsilon/t],~ \text{and}~ 1-q\in [0, \varepsilon/t]\right\},.
				\end{split}
			\end{equation*}

		In the case where $(p, q) \in \mathfrak B^1(\varepsilon/t)$,
		we exploit that $\phi_0^{-1}\, \widehat p_{(q-p)t}^{\mu, F}(V\cdot  \varPhi_{(1-q)t})
		=p_{(q-p)t}^{\phi_0}(V \cdot \varPhi_{(1-q)\,t}\cdot \phi_0^{-1})$
		is uniformly upper-bounded by $c_\varepsilon^2\int_E\phi_0{\rm d}\mathfrak m \cdot\|V\|_{L^1(E; \phi_0^2\d\mathfrak{m})}$. Then, the dependency of $q$ in the integrant has been removed and the integral over $q$ can be upper-bounded by $1$. Integrating over $p$ leads to the following upper-bound:
		\begin{equation}\label{eq_B1_eps}
			\begin{split}
				&\iint_{\mathfrak B^1(\varepsilon/t)}\varPhi_t(x)^{-1}
				\widehat p^{\mu, F}_{pt}\left(V\cdot \widehat p^{\mu, F}_{(q-p)t}\left(V\cdot \varPhi_{(1-q) t}\right)\right)(x)\, {\d}q\, {\d}p \\
				&\quad \le \frac{2}{\int_E \phi_0{\d}\mathfrak{m}}\iint_{\mathfrak B^1(\varepsilon/t)}\phi_0(x)^{-1}
				\widehat p^{\mu, F}_{pt}\left(V\cdot \widehat p^{\mu, F}_{(q-p)t}\left(V\cdot \varPhi_{(1-q) t}\right)\right)(x)\, {\d}q\, {\d}p \\
				&\quad \le \frac{2c_\varepsilon^2}t \cdot\|V\|_{L^1(E; \phi_0^2\d\mathfrak{m})}\cdot {\mathbb{E}}^{\phi_0}_x\left[A_\varepsilon^V\right].
			\end{split}
		\end{equation}
		When $(p, q) \in \mathfrak B^2(\varepsilon/t)$, we look for an upper-bound of the $L^1(E;\phi_0^2\mathfrak{m})$-norm 
		of $\phi_0^{-1}\,V\cdot \varPhi_{(1-q) t}$ integrated over all $q\in [1-\varepsilon/t, 1]$, with the change of variable $r = (1-q)t$:
		\begin{equation*}
			\begin{split}
				\left\|\int_{1-\varepsilon/t}^1 \phi_0^{-1} V\cdot \varPhi_{(1-q)t} \d q\right\|_{L^1(E;\phi_0^2\mathfrak{m})}&=\frac{1}{t}\int_E \int_0^\varepsilon \phi_0(y)^{-1}V(y)\varPhi_r(y) \phi_0^2(y)\,{\d}r\,\mathfrak{m}(\d y) \\
				&=\frac{1}{t}\int_E \int_0^\varepsilon V(y)\, p_r^{\phi_0}(\phi_0^{-1})(y)\, \phi_0^2(y)\,{\d}r\, \mathfrak{m}(\d y) \\
				&=\frac{1}{t}\int_E \int_0^\varepsilon p_r^{\phi_0}(V)(y)\, \phi_0(y)\,{\d}r\, \mathfrak{m}(\d y) \\
				&=\frac{1}{t}\int_E {\mathbb E}_y^{\phi_0}\left[A_\varepsilon^{V}\right]\, \phi_0(y)\mathfrak{m}(\d y),
			\end{split}
		\end{equation*}
		where we used the symmetry of $p^{\phi_0}_{r}$ for the third line.
		Since $(q-p) t\ge \varepsilon$, this entails that
		\begin{equation}\label{eq_sep_Bound}
			\left\|\phi_0^{-1}\,\widehat p^{\mu, F}_{(q-p)t}\left(\int_{1-\varepsilon/t}^{1} V\cdot \varPhi_{(1-q) t}\, {\d}q\right)\right\|_{L^\infty(E; \phi_0^2\mathfrak m)}
			\le \frac{c_\epsilon}{t}\, \int_E {\mathbb{E}}^{\phi_0}_y\left[A_\varepsilon^{V}\right] \phi_0(y)\mathfrak m({\rm d}y)\,.
		\end{equation}
		By linearity and since $pt\ge \varepsilon$, we thus deduce that
		\begin{equation}\label{eq_B2_eps}
			\begin{split}
				&\iint_{\mathfrak B^2(\varepsilon/t)}\varPhi_t(x)^{-1}
				\widehat p^{\mu, F}_{pt}\left(V\cdot \widehat p^{\mu, F}_{(q-p)t}\left(V\cdot \varPhi_{(1-q) t}\right)\right)(x)\, {\d}q\, {\d}p \\
				&\quad \le \frac{2}{\int_E \phi_0{\d}\mathfrak{m}}\iint_{\mathfrak B^2(\varepsilon/t)}\phi_0(x)^{-1}
				\widehat p^{\mu, F}_{pt}\left(V\cdot \widehat p^{\mu, F}_{(q-p)t}\left(V\cdot \varPhi_{(1-q) t}\right)\right)(x)\, {\d}q\, {\d}p \\
				&\quad \le \frac{C_1}t\cdot \|V\|_{L^1(E; \phi_0^2\d\mathfrak{m})}\int_E {\mathbb{E}}^{\phi_0}_y\left[A_\epsilon^{V}\right] \phi_0(y)\,\mathfrak m({\rm d}y),
			\end{split}
		\end{equation}
		where $C_1:=2c_\varepsilon^2/\int_E \phi_0 {\d}\mathfrak{m}$. 
		When $(p, q) \in \mathfrak B^3(\varepsilon/t)$, the estimate \eqref{eq_sep_Bound} is derived similarly as when $(p, q) \in \mathfrak B^2(\varepsilon/t)$.
		We can then handle the remainder as when $(p, q) \in \mathfrak B^1(\varepsilon/t)$, leading to:
		\begin{equation}\label{eq_B3_eps}
			\begin{split}
				&\iint_{\mathfrak B^3(\varepsilon/t)}\varPhi_t(x)^{-1}
				\widehat p^{\mu, F}_{pt}\left(V\cdot \widehat p^{\mu, F}_{(q-p)t}\left(V\cdot \varPhi_{(1-q) t}\right)\right)(x)\, {\d}q\, {\d}p \\
				&\quad \le \frac{C_2}{t}\cdot \int_0^{\varepsilon/t} p_{pt}^{\phi_0}(V)(x) {\d}p \cdot \int_E {\mathbb{E}}^{\phi_0}_y\left[A_\varepsilon^{V, 0}\right] \phi_0(y)\,\mathfrak m({\rm d}y) \\
				&\quad 	= \frac{C_2}{t^2}\cdot{\mathbb{E}}^{\phi_0}_x\left[A_\varepsilon^{V}\right]\cdot\int_E {\mathbb{E}}^{\phi_0}_y\left[A_\varepsilon^{V}\right] \phi_0(y)\,\mathfrak m({\rm d}y),
			\end{split}
		\end{equation}
		where $C_2:=2c_\varepsilon/\int_E \phi_0 {\d}\mathfrak{m}$. 
		When $(p, q) \in \mathfrak B^4(\varepsilon/t)$,
		we first note that $(1-q)t\ge \varepsilon$,
		which means we can get a uniform upper-bound $c_{\varepsilon}\int_E \phi_0{\d}\mathfrak{m}$ of $\phi_0^{-1}\varPhi_{(1-q)t}$ by \eqref{eq_phi_ratio}.
		Besides, $0\le p <q \le 2\varepsilon/t$, so that using the change of variables $r=pt$, 
	$s = qt$: 
		\begin{equation}\label{eq_def_A2eps}
			\begin{split} 
				&\iint_{\mathfrak B^4(\varepsilon/t)}\phi_0(x)^{-1}
				\widehat p^{\mu, F}_{pt}\left(V\cdot \widehat p^{\mu, F}_{(q-p)t}\left(V\cdot \phi_0\right)\right)(x)\, {\d}q\, {\d}p \\
				&\quad \le \frac{1}{t^2}\int_0^{\varepsilon} \int_r^{2\varepsilon} p^{\phi_0}_r\left(V\cdot p^{\phi_0}_{s-r}(V)\right)(x)\, {\d}s \, {\d}r
					\le \frac{1}{2t^2}{\mathbb{E}}^{\phi_0}_x\left[\left(A_{2\epsilon}^{V}\right)^2\right].
			\end{split}
		\end{equation}
		It leads us to the following upper-bound:
		\begin{equation}\label{eq_B4_eps}
			\iint_{\mathfrak B^4(\varepsilon/t)}\varPhi_t(x)^{-1}
			\widehat p^{\mu, F}_{pt}\left(V\cdot \widehat p^{\mu, F}_{(q-p)t}\left(V\cdot \varPhi_{(1-q) t}\right)\right)(x)\, {\d}q\, {\d}p
			\le \frac{c_\epsilon}{t^2}\cdot  {\mathbb{E}}^{\phi_0}_x\left[\left(A_{2\epsilon}^{V}\right)^2\right]\,.
		\end{equation}
		When $(p, q) \in \mathfrak B^5(\varepsilon/t)$,
		$(1-q)t\ge \epsilon$, so that $c_{\varepsilon}\int_E \phi_0{\d}\mathfrak{m}$ is again a uniform upper-bound of $\phi_0^{-1}\varPhi_{(1-q)t}$. We look for an upper-bound of the $L^1(E;\phi_0^2\mathfrak{m})$-norm 
		of $\phi_0^{-1} V\cdot \widehat p^{\mu, F}_{(q-p)t}(V\cdot \phi_0)$ integrated over all $q\in [p, p+\varepsilon/t]$, with the change of variable $r = (q-p)t$:
		\begin{equation}\label{eq_L1_B5}
			\begin{split}
				\left\|\int_p^{p+\varepsilon/t}\phi_0^{-1} V\cdot \widehat{p}_{(q-p)t}^{\mu,F}(V\cdot \phi_0)\,{\d}q\right\|_{L^1(E;\phi_0^2\mathfrak{m})}&=\frac{1}{t}\int_0^{\varepsilon}\int_E	V(y)\, p^{\phi_0}_{r}(V)(y)\,\phi_0^2(y) \mathfrak m({\rm d}y)\,{\rm d}r.
			\end{split}
		\end{equation}
		On the other hand, since the measure $\phi_0^2 \mathfrak m$ is invariant by the semigroup $\{p^{\phi_0}_t\}_{t\ge 0}$,
		\begin{align*}
			\int_0^{\varepsilon}\int_E	V(y)\, p^{\phi_0}_{r}(V)(y)\,\phi_0^2(y)\, \mathfrak m({\rm d}y)\,{\rm d}r &= \frac{1}{\varepsilon}\int_0^{\varepsilon}\int_0^{\varepsilon}\int_E V(y)\, p^{\phi_0}_{r}(V)(y)\,\phi_0^2(y)\, \mathfrak m({\rm d}y)\,{\rm d}r{\d}s \\
			&\le \frac{1}{\varepsilon}\int_0^{2\varepsilon}\int_0^{2\varepsilon-s}\int_E p_s^{\phi_0}\left(V\cdot p^{\phi_0}_{r}V\right)(y)\phi_0^2(y)\, \mathfrak m({\rm d}y)\,{\rm d}r\, {\d}s \\
			&=\frac{1}{2\varepsilon}\int_E {\mathbb E}_y^{\phi_0}\left[\left(A_{2\varepsilon}^{V}\right)^2\right]\phi_0^2(y)\mathfrak{m}(\d y).
		\end{align*}
		We thus arrive at the following upper-bound by considering $pt\ge \varepsilon$:
		\begin{equation}\label{eq_B5_eps}
			\begin{split}
				&\iint_{\mathfrak B^5(\varepsilon/t)}\varPhi_t(x)^{-1}	\widehat p^{\mu, F}_{pt}\left(V\cdot \widehat p^{\mu, F}_{(q-p)t}\left(V\cdot \varPhi_{(1-q) t}\right)\right)(x)\, {\d}q\, {\d}p \\
				&\quad \le 2c_{\varepsilon} \iint_{\mathfrak B^5(\varepsilon/t)}\phi_0(x)^{-1}\widehat p^{\mu, F}_{pt}\left(V\cdot \widehat p^{\mu, F}_{(q-p)t}\left(V\cdot \phi_0\right)\right)(x)\, {\d}q\, {\d}p \\
				&\quad \le 2c_{\varepsilon}^2\int_{\varepsilon/t}^{1-\varepsilon/t}\left\|\int_p^{p+\varepsilon/t}\phi_0^{-1}\cdot V\cdot \widehat{p}_{(q-p)t}^{\mu,F}(V\cdot \phi_0)\,{\d}q\right\|_{L^1(E;\phi_0^2\mathfrak{m})}{\d}p \\
				&\quad \le \frac{C_3}{t}\cdot \int_E {\mathbb{E}}^{\phi_0}_y\left[\left(A_{2\varepsilon}^{V}\right)^2\right] \phi_0(y) \mathfrak m({\rm d}y),
			\end{split}
		\end{equation}
		where $C_3:=c_\varepsilon^2\|\phi_0\|_\infty /\varepsilon$.
		Lastly, when $(p, q) \in \mathfrak B^6(\varepsilon/t)$,
		we adapt \eqref{eq_L1_B5} without the replacement of $\varPhi_{(1-q)t}$ by $\phi_0$, together with the change of variables $r = (1-p)t$ then $s=(1-q)t$:
			\begin{align*}
				&\left\|
				\int_{1-\varepsilon/t}^1
				\int_{1-2\varepsilon/t}^q \phi_0^{-1} V\cdot \widehat{p}_{(q-p)t}^{\mu,F}(V\cdot \varPhi_{(1-q)t})\, {\d}p\,{\d}q\right\|_{L^1(E;\phi_0^2\mathfrak{m})}
				\\&\qquad=\frac{1}{t^2}\int_0^{\varepsilon}\int_s^{2\varepsilon}\int_E	V(y)\, p^{\phi_0}_{r-s}\left(V\cdot p_{s}^{\phi_0}(\phi_0^{-1})\right)(y)\,\phi_0^2(y) \mathfrak m({\rm d}y)\,{\rm d}r\,{\rm d}s
				\\&\qquad = \frac{1}{t^2}
				\int_0^{\varepsilon}\int_s^{2\varepsilon}\int_E	p^{\phi_0}_{s}\left(V p^{\phi_0}_{r-s}(V)\right)(y)\,\phi_0(y) \mathfrak m({\rm d}y)\,{\rm d}s\,{\rm d}r\,
				\\&\qquad \le \frac{1}{2 t^2} \int_E {\mathbb{E}}^{\phi_0}_y\left[\left(A_{2\varepsilon}^{V}\right)^2\right] \phi_0(y) \mathfrak m({\rm d}y)\,.
			\end{align*}
			Since $pt\ge \epsilon$, we thus conclude to the following upper-bound:
			\begin{equation}\label{eq_B6_eps}
				\iint_{\mathfrak B^6(\varepsilon/t)}\varPhi_t(x)^{-1}
				\widehat p^{\mu, F}_{pt}\left(V\cdot \widehat p^{\mu, F}_{(q-p)t}\left(V\cdot \varPhi_{(1-q) t}\right)\right)(x)\, {\d}q\, {\d}p
				\le \frac{C_2}{2 t^2}\,\int_E {\mathbb{E}}^{\phi_0}_y\left[\left(A_{2\varepsilon}^{V}\right)^2\right] \phi_0(y) \mathfrak m({\rm d}y)\,,
			\end{equation}
			where we recall $C_2 = 2c_\varepsilon/\int_E \phi_0 {\d}\mathfrak{m}$.
	\vskip 0.1cm
		We now recall Lemma~\ref{lem_AF_id},
		in which we demonstrated
		that $V\in \mathcal S^1_D(X)$
		entails property \eqref{eq_bd_AfGphi0}.
		We see for each of the terms in \eqref{eq_B1_eps}, \eqref{eq_B2_eps}, \eqref{eq_B3_eps}, \eqref{eq_B4_eps}, \eqref{eq_B5_eps}, \eqref{eq_B6_eps}
		that they tend to 0 as $t\to \infty$.
		We assumed the function $V$ to be nonnegative to establish these convergences, yet they directly extend for any signed function $V$ with the bound given by $|V|$ instead of $V$.
		We thus conclude Lemma~\ref{Moment2} by additionally recalling \eqref{eq_lim_wo_bnd}.
\end{proof}

\begin{proof}[Proof of Lemma~\ref{JumpWLLN}]
We first assume that $G$ is nonnegative.Thanks to the expression \eqref{tg2}, 
		we have for  $x \in E$, 
		\begin{equation*}
			\begin{split}
				&{\mathbb{E}}_{x|t}^{\mu, F}\left[\left(\frac{1}{t}A_t^G\right)^2\right] \\
				&\quad = \frac{1}{t}{\mathbb{E}}_{x|t}^{\mu, F}\left[A_t^{G^2}\right]	
				+	\int_0^1\int_p^1  
				\varPhi_t^{-1}(x)\,\widehat p^{\mu, F}_s \left(N\Big[G e^{-F}\cdot \widehat p^{\mu, F}_r\big(N\left[G e^{-F}\cdot \varPhi_{t-s-r}\right]\big)\Big]\right)(x)\,{\d}p\,{\d}q\,.
			\end{split}
		\end{equation*}
		The first term of the above right-hand side  converges to $0$ as $t \to \infty$, due to Theorem~\ref{JumpErgodic},
		Lemma~\ref{lem_AF_id} and \eqref{eq_prop_L1phi}, since $N[G^2], N[(G^2)^T] \in L^1(E;\mathfrak{m})\cap \mathcal S_D^1(X)$.
		The second term is handled similarly as for Lemma~\ref{Moment2}.
		The integrant $\Xi^G_t(p, q, x)$ is upper-bounded for any $t \ge \max\{\varepsilon/p, \varepsilon/(q-p), \varepsilon/(1-q)\}$:
		\begin{equation*}
			\Xi^G_t(p, q, x)
			= \frac{\phi_0(x)}{\varPhi_t(x)} p^{\phi_0}_{pt}\left(N^{\phi_0}\left[G\cdot  p^{\phi_0}_{(q-p)t}\left(N^{\phi_0}[G\,\phi_0^{-1}\, \varPhi_{(1-q) t}]\right)\right]\right)(x)
			\le 2 c_\varepsilon^3 \cdot \|N^{\phi_0}[G]\|^2_{L^1(E; \phi_0^2\mathfrak{m})}\,.
		\end{equation*}
		By the dominated convergence theorem,
		the assertion of Lemma \ref{QED1}$(4)$ with $g=1$ and $K=G$ implies the following convergence:
		\begin{equation}\label{eq_limG_wo_bnd}
			\left(\iint_{E\times E}G(y,z){\mathcal J}_{\phi_0}({\rm d}y{\rm d}z)\right)^2
			\\	= 2 \lim_{t\to \infty} \int_{\varepsilon/t}^{1-\varepsilon/t}\int_{
				p + \varepsilon/t}^{1-\varepsilon/t}\Xi^G_t(p, q, x)\, {\d}p\, {\d}q\,.
		\end{equation}
		The convergence is extended for signed functions $G$ by exploiting again the identity 
		$(a-b)^2 = 2a^2 + 2b^2 - (a+b)^2$,
		and linearity properties of $N$ and $\{\widehat p_s^{\mu, F}\}_{s\ge 0}$.
		\vskip 0.1cm
		We exploit the same decomposition of the boundary terms as for Lemma~\ref{Moment2} and similar arguments to deduce the following upper-bounds,
		where we first assume for simplicity that $G$ is nonnegative:
		\begin{align}
			\label{eq_B1_G}
			\iint_{\mathfrak B^1(\varepsilon/t)}
			\Xi^G_t(p, q, x)\, {\d}p\, {\d}q
			&\le \frac{2 c_\varepsilon^2}t \cdot\|N^{\phi_0}[G]\|_{L^1(E; \phi_0^2\d\mathfrak{m})}\cdot {\mathbb{E}}^{\phi_0}_x\left[A_\varepsilon^{G}\right],
			\\ \label{eq_B2_G}
			\iint_{\mathfrak B^2(\varepsilon/t)}
			\Xi^G_t(p, q, x)\, {\d}p\, {\d}q
			&\le \frac{C_1}t\, \cdot\|N^{\phi_0}[G]\|_{L^1(E; \phi_0^2\d\mathfrak{m})}\cdot \int_E{\mathbb{E}}^{\phi_0}_y\left[A_\varepsilon^{G^T}\right]\phi_0(y) \mathfrak m({\rm d}y)\,,
			\\ \label{eq_B3_G}
			\iint_{\mathfrak B^3(\varepsilon/t)}
			\Xi^G_t(p, q, x)\, {\d}p\, {\d}q
			&\le \frac{C_2}t\,\cdot {\mathbb{E}}^{\phi_0}_x\left[A_\varepsilon^{G}\right]\cdot\int_E{\mathbb{E}}^{\phi_0}_y\left[A_\varepsilon^{G^T}\right]\phi_0(y) \mathfrak m({\rm d}y)\,,
			\\ \label{eq_B4_G}
			\iint_{\mathfrak B^4(\varepsilon/t)}
			\Xi^G_t(p, q, x)\, {\d}p\, {\d}q
			&\le \frac{c_\varepsilon}{t^2}\,\int_E\phi_0{\rm d}\mathfrak m \cdot  {\mathbb{E}}^{\phi_0}_x\left[\left(A_{2\varepsilon}^{G}\right)^2\right]\,,
			\\ \label{eq_B5_G}
			\iint_{\mathfrak B^5(\varepsilon/t)}
			\Xi^G_t(p, q, x)\, {\d}p\, {\d}q
			&\le \frac{C_3}{t}\cdot \int_E {\mathbb{E}}^{\phi_0}_y\left[\left(A_{2\varepsilon}^{G}\right)^2\right]\phi_0(y)\,\mathfrak m(\d y)\,,
			\\ \label{eq_B6_G}
			\iint_{\mathfrak B^6(\varepsilon/t)}
			\Xi^G_t(p, q, x)\, {\d}p\, {\d}q
			&\le \frac{C_2}{2 t^2}\, \cdot \int_E {\mathbb{E}}^{\phi_0}_y\left[\left(A_{2\varepsilon}^{G^T}\right)^2\right]\phi_0(y)\,\mathfrak m(\d y)\,.
		\end{align}
		Let us provide some further details regarding the proof of   \eqref{eq_B6_G}, while leaving 
		the adaptation of \eqref{eq_B1_eps}, \eqref{eq_B2_eps}, \eqref{eq_B3_eps}, \eqref{eq_B4_eps}, \eqref{eq_B5_eps}
		into \eqref{eq_B1_G} -- \eqref{eq_B5_G}
		to the interested reader.
		The main distinction of   \eqref{eq_B6_G} as compared to \eqref{eq_B6_eps} is that we need additional steps that involve \eqref{eq_sym_Jphi}.
		In particular, for any $r, s\in [0, \varepsilon]$:
		\begin{equation*}
			\begin{split}
				&\int_E	 N^{\phi_0}\left[G\cdot  p^{\phi_0}_{r}\left(N^{\phi_0}[G\cdot p^{\phi_0}_{s}(\phi_0^{-1})]\right)\right](y)\cdot  
				\phi_0^2(y) \mathfrak m({\rm d}y)
				\\&\qquad= \int_E	
				p^{\phi_0}_{r}\left(N^{\phi_0}[G\cdot p^{\phi_0}_{s}(\phi_0^{-1})]\right)(y)\cdot
				N^{\phi_0}[G^T](y)\cdot 
				\phi_0^2(y) \mathfrak m({\rm d}y)
				\\&\qquad=\int_E	
				N^{\phi_0}[G\cdot p^{\phi_0}_{s}(\phi_0^{-1})](y)\cdot
				p^{\phi_0}_{r}\left(N^{\phi_0}[G^T]\right)(y)\cdot 
				\phi_0^2(y) \mathfrak m({\rm d}y)
				\\&\qquad= \int_E	
				p^{\phi_0}_{s}\left(N^{\phi_0}\left[G^T\cdot p^{\phi_0}_{r}\left(N^{\phi_0}[G^T]\right)\right]\right)(y)\cdot
				\phi_0(y) \mathfrak m({\rm d}y)\,.
			\end{split}
		\end{equation*}
		Hence
		\begin{equation*}
			\left\|\int_{1-\varepsilon/t}^1
			\int_{1-2\varepsilon/t}^{q}N^{\phi_0}\left[G\cdot p^{\phi_0}_{(q-p)t}\left(N^{\phi_0}[G\,\phi_0^{-1}\, \varPhi_{(1-q)t}]\right)\right]\,{\rm d}p\,{\rm d}q\right\|_{L^1(E; \phi_0^2  \mathfrak m)}
			\le \frac1{t^2}\int_E {\mathbb{E}}^{\phi_0}_y\left[\left(A_{2\varepsilon}^{G^T}\right)^2\right]\phi_0(y)\,\mathfrak m(\d y)\,.
		\end{equation*}
		\eqref{eq_B6_G} is then deduced similarly as \eqref{eq_B6_eps}
		by exploiting that $pt\ge \varepsilon$.
		\vskip 0.1cm
		By Lemma~\ref{lem_AF_id}  and \eqref{eq_prop_L1phi},
		each of the terms in \eqref{eq_B1_G} -- \eqref{eq_B6_G}
		tend to 0 as $t\to \infty$,
		since $N[G], N[G^T] \in L^1(E;\mathfrak{m})\cap \mathcal S_D^1(X)$.
		This property extends to any signed function $G$ 
		by considering $|G|$ instead of $G$ in the above convergence.
		We complete the proof of Lemma~\ref{Moment2} by also recalling the dominant term in \eqref{eq_limG_wo_bnd}.
	
\end{proof}


\subsection*{Acknowledgments}
The authors would like to thank the anonymous referees, an Associate
Editor and the Editor for their constructive comments that improved the
quality of this paper.


\vskip 1.0cm
\noindent
Daehong Kim \\
Department of Mathematics and Engineering, \\
Graduate School of Science and Technology,
Kumamoto University, \\
Kumamoto, 860-8555, Japan \\
{\tt daehong@gpo.kumamoto-u.ac.jp}
\vskip 0.6cm
\noindent
Takara Tagawa \\
Department of Mathematics and Engineering, \\
Graduate School of Science and Technology,
Kumamoto University, \\
Kumamoto, 860-8555, Japan \\
{\tt 236d9321@st.kumamoto-u.ac.jp}
\vskip 0.6cm
\noindent
Aur\'elien Velleret \\
LaMME, UMR CNRS 8071, \\
Université d'Evry Val d'Essonne, Paris Saclay, \\
23 Boulevard de France, 91037 Evry, France\\
{\tt velleret@phare.normalesup.org}

\end{document}